\documentclass{article}

\usepackage{amsmath,amsthm,amssymb,epsfig,xargs,hyperref,cleveref}
\usepackage{dsfont}
\usepackage{colortbl}
\usepackage[numbers]{natbib}
\usepackage{algorithm}
\usepackage{algpseudocode}
\usepackage{enumitem}

\usepackage{tikz}
\usetikzlibrary{matrix,calc}
\usetikzlibrary{decorations.markings,decorations.pathreplacing}
\usetikzlibrary{patterns}

\usepackage{booktabs}

\usepackage[margin=2.5cm]{geometry}

\newtheorem{theorem}{Theorem}[section]
\newtheorem{lemma}[theorem]{Lemma}
\newtheorem{corollary}[theorem]{Corollary}

\tikzset{square matrix/.style={
	matrix of nodes,
	column sep=-\pgflinewidth, row sep=-\pgflinewidth,
	nodes={draw,
	minimum height=15pt,
	anchor=center,
					text width=13pt,
					align=center,
					inner sep=0pt
				},
				},
				square matrix/.default=2cm
}

\tikzset{square freq/.style={
			matrix of nodes,
			column sep=-\pgflinewidth, row sep=-\pgflinewidth,
			nodes={draw,
					minimum height=25pt,
					anchor=center,
					text width=25pt,
					align=center,
					inner sep=0pt
				},
		},
	square freq/.default=2cm
}

\tikzset{square freq small/.style={
			matrix of nodes,
			column sep=-\pgflinewidth, row sep=-\pgflinewidth,
			nodes={draw,
					minimum height=20pt,
					anchor=center,
					text width=20pt,
					align=center,
					inner sep=0pt
				},
		},
	square freq small/.default=2cm
}

\definecolor{tblue}{HTML}{88CCEE}
\definecolor{tred}{HTML}{CC6677}
\definecolor{tgreen}{HTML}{44AA99}
\definecolor{tyellow}{HTML}{DDCC77}
\definecolor{topurple}{HTML}{9B83C7}

\colorlet{t0}{white!0}
\colorlet{t1}{tblue}
\colorlet{t2}{tred}
\colorlet{t3}{tgreen}
\colorlet{t4}{tyellow}
\colorlet{t5}{topurple}

\renewcommand{\leq}{\leqslant}
\renewcommand{\geq}{\geqslant}
\renewcommand{\le}{\leqslant}
\renewcommand{\ge}{\geqslant}

\renewcommand\emptyset{\varnothing}
\newcommand\eref[1]{{\ensuremath(\ref{#1})}}
\newcommand\sref[1]{\S$\ref{#1}$}
\newcommand\fref[1]{Figure~$\ref{#1}$}

\newcommand\Tref[1]{Table~$\ref{#1}$}
\newcommand\lref[1]{Lemma~$\ref{#1}$}

\newcommand{\BB}{\mathcal{B}}
\newcommand{\LL}{\mathcal{L}}
\newcommand{\OO}{\mathcal{O}}
\newcommand{\PP}{\mathcal{P}}
\newcommand{\RR}{\mathcal{R}}
\newcommand{\MS}{\mathcal{S}}

\makeatletter
\g@addto@macro\bfseries{\boldmath}
\makeatother

\newcommand{\mols}[2]{\ensuremath{#1\text{-}\mathrm{MOLS}(#2)}}
\newcommand{\molsb}[2]{\mols{(#1)}{#2}}
\newcommand{\mofs}[2]{\ensuremath{#1\text{-}\mathrm{MOFS}(#2)}}
\newcommand{\net}[2]{\ensuremath{#1\text{-}\mathrm{net}(#2)}}
\newcommand{\netb}[2]{\net{(#1)}{#2}}

\newcommand{\FField}[1]{\ensuremath{\mathbb{F}_{#1}}}
\newcommand{\bvec}[1]{\ensuremath{\mathbf{#1}}}
\newcommand{\aut}[1]{\ensuremath{\text{Aut}(#1)}}

\newcommand{\zto}[1]{\ensuremath{\{0,\dots,#1\}}}

\newcommand{\ztosublist}[2]{\ensuremath{(#1_0,\dots,#1_{#2})}}

\newcommand{\otosublist}[2]{\ensuremath{(#1_0,\dots,#1_{#2-1})}}

\newcommand{\rel}{\RR}
\newcommand{\relodd}{\RR_\OO}

\newcommand{\moa}{\text{MOA}}

\crefrangelabelformat{equation}{(#3#1#4) --~(#5#2#6)}

\title{Pairs of MOLS of order ten satisfying non-trivial relations}

\author{Michael J. Gill\thanks{Research supported by an Australian Government
    Research Training Program (RTP) Scholarship} \ \
  and \ Ian M. Wanless
  \\
\small School of Mathematics\\[-0.75ex]
\small Monash University\\[-0.75ex]
\small Vic 3800, Australia\\
\small\tt \{michael.gill,ian.wanless\} \ @monash.edu
}

\begin{document}

\date{}

\maketitle

\begin{abstract}
A relation on a $\net{k}{n}$ (or, equivalently, a set of $k-2$ mutually
orthogonal Latin squares of order $n$) is an $\FField{2}$ linear dependence
within the incidence matrix of the net. Dukes and Howard (2014) showed that any
\net{6}{10} satisfies at least two non-trivial relations, and classified the
relations that could appear in such a net. We find that, up to equivalence,
there are $18\,526\,320$ pairs of MOLS satisfying at least one non-trivial
relation. None of these pairs extend to a triple. We also rule out one other
relation on a set of $3$-MOLS from Dukes and Howard's classification.
\end{abstract}

\newdimen\digitwidth \setbox0=\hbox{\rm0} \digitwidth=\wd0
\catcode`@=\active \def@{\kern\digitwidth}

\section{Introduction}

A \emph{partial linear space} $(P, L)$ consists of a set $P$ of \emph{points}
and a set $L$ of subsets of $P$ called \emph{lines}, such that distinct lines
share at most one point. A line $\ell\in L$ and a point $p\in P$ are
\emph{incident} if $p\in \ell$. Two lines $\ell_1$ and $\ell_2$ are
\emph{parallel} if $\ell_1=\ell_2$ or $\ell_1 \cap\ell_2 = \emptyset$.
Similarly, two lines $\ell_1$ and $\ell_2$ are \emph{orthogonal} if
$|\ell_1\cap\ell_2| = 1$. A \net{k}{n} is a partial linear space $N = (P, L)$
consisting of a set $P$ of $n^2$ points and a set $L$ of $kn$ lines, such that
each line is incident with $n$ points and each point is incident with $k$
lines. The lines of $N$ should partition into $k$ \emph{parallel classes} with
each parallel class containing $n$ pairwise parallel lines. Every pair of lines
that are not in the same parallel class must be orthogonal. Two $k$-nets, $N_1
= (P_1, L_1)$ and $N_2 = (P_2,L_2)$, of order $n$ are isomorphic if there
exists a bijection $\phi: P_1 \rightarrow P_2$ such that $\ell \in L_1$ if and
only if $\phi(\ell) \in L_2$. A \net{k'}{n} $N_1 = (P,L_1)$ is a \emph{subnet}
of a \net{k}{n} $N_2 = (P,L_2)$ if $L_1 \subseteq L_2$.

Let $N = (P,L)$ be a \net{k}{n}. The \emph{incidence matrix} $M(N)$ is an
$n^2\times nk$ array over the finite field $\FField{2}$ with rows and columns
indexed by the points and lines of $N$, respectively. The entry $M_{i,j} = 1$
if the point $i$ is contained in line $j$ and $M_{i,j} = 0$ otherwise. For a
line $\ell$ in the net, we define $\alpha(\ell)$ to be the column vector in
$M(N)$ corresponding to the line $\ell$. A \emph{relation} $\rel \subseteq L$
is a set of lines such that $\sum_{\ell\in\rel} \alpha(\ell) = \bvec{0}$, the
all zero $\FField{2}$-vector.  In other words, a relation is an
$\FField{2}$-linear dependence among the columns of $M(N)$.  It is immediate
from this definition that the symmetric difference of two relations is a
relation. Every parallel class $\Pi$ has $\sum_{\ell\in \Pi}\alpha(\ell) =
\bvec{1}$, the all one $\FField{2}$-vector, so the union of any pair of
distinct parallel classes is a relation. Any relation that is the union of any
number of parallel classes is considered \emph{trivial}. 

Each relation on $N$ corresponds to a null vector of $M(N)$. We say that a set
of relations is \emph{linearly independent} if the corresponding set of null
vectors is linearly independent. We define the \emph{dimension} of $N$, written
$\dim(N)$, to be the rank of $M(N)$.  Throughout, we will use
$\Pi_0,\Pi_1,\dots,\Pi_{k-1}$ to denote the parallel classes of a $k$-net $N$.
The set $\{\Pi_0\cup\Pi_1,\Pi_0\cup\Pi_2,\dots,\Pi_0\cup\Pi_{k-1}\}$ is a set
of $k-1$ linearly independent relations on $N$ which span the trivial
relations. It follows that
\begin{equation}\label{e:dimNbound}
	\dim(N) \le nk-k+1,
\end{equation}
with equality if and only if there are no non-trivial relations on $N$.

Let $N = (P, L)$ be a \net{k}{n} and let $\rel$ be a relation on $N$. We say
that a line $\ell \in L$ is \emph{relational} if $\ell\in\rel$. The weight
$w(p)$ of a point $p$ is the number of relational lines incident to $p$ in $N$.
The fact that $\rel$ is a relation is equivalent to the statement that
$w(p)\equiv0\pmod{2}$ for all points $p\in P$.  The number of relational lines
in a parallel class $\Pi$ is the \emph{weight} of $\Pi$. The \emph{type} of a
relation indicates the weight of each parallel class in $N$. For simplicity, we
will often use power notation, with $a^b$ indicating that $b$ parallel classes
have weight $a$. For example, a relation on a \net{6}{n} with $3$ parallel
classes containing $2$ relational lines, $2$ parallel classes containing $4$
relational lines and $1$ parallel class containing $6$ relational lines has
type $(2,2,2,4,4,6)$ or $2^34^26$.

\citet*{dukes_group_2014}, building on earlier work by
\citet*{stinson_short_1984} and \citet*{dougherty_coding_1994-1}, proved the
following necessary condition on the set of relations on a \net{6}{10}, if such
a net exists.

\begin{theorem}\label{t:DH}
  Any \net{6}{10} satisfies at least two linearly independent non-trivial
  relations.
\end{theorem}

Additionally, \citet{dukes_group_2014} showed that any non-trivial relation on
a \net{6}{10} has one the following types (up to taking disjoint union with a
trivial relation): $4^4$, $2^346$, $2^24^3$, $24^36$, $4^5$, $2^6$, $2^44^2$,
$2^34^26$, $2^24^4$, $24^46$, $4^6$. In this paper we report on two
computations on nets of order 10.  The first found all 4-nets that satisfy a
non-trivial relation. None of them extends to a 5-net. The second computation
showed that no 5-net satisfies a relation of type $2^346$. We are therefore
able to rule out a 6-net satisfying either of the first two types of relations
in Dukes and Howard's catalogue. 

The structure of this paper is as follows.  In \Cref{s:background} we give the
necessary background and definitions, including the interpretation of our
problem in terms of mutually orthogonal Latin squares.  In \Cref{s:templates},
we prove a general result on the structure of relations in a \net{6}{n} for
even $n$.  We also give some structural lemmas for nets with relations of type
$4^4$ and $2^346$. In \Cref{s:refining}, we describe algorithms that generate a
catalogue of nets that satisfy a given relation.
\Cref{section:nets4444,s:44222} report on the computational results we obtained
using those algorithms for relations of type $4^4$ and $2^346$, respectively.
For the purposes of crosschecking, all computational results reported in this
paper were obtained independently by the two authors.

\section{Background and Definitions}\label{s:background}

A \emph{frequency square} of order $n$, with $s$ symbols, and frequencies
$\ztosublist{\lambda}{s-1}$ is an $n\times n$ array with a symbol multiset $S$
containing $\lambda_i$ copies of symbol $i$, such that each row and each column
is a permutation of $S$. A frequency square is \emph{balanced} if each symbol
has the same multiplicity in $S$. A balanced frequency square for which each
symbol occurs once in $S$ is a \emph{Latin square}.

We index the rows and columns of $n\times n$ matrices using
$\{0,1,\dots,n-1\}$.  Two frequency squares $F_1$ and $F_2$ with frequencies
$\ztosublist{\lambda}{s-1}$ and $\ztosublist{\mu}{t-1}$ are \emph{orthogonal}
if for each $\ell \in\zto{s-1}$ and $m \in \zto{t-1}$ there are
$\lambda_\ell\mu_m$ ordered pairs $(F_1[i,j], F_2[i,j]) = (\ell,m)$ with $i,j
\in \zto{n-1}$. A set of $k$ frequency squares of order $n$ is \emph{mutually
orthogonal} if every pair of squares in the set is orthogonal. In this case we
write that the frequency squares form a set of \mofs{k}{n}. If every square in
the set of \mofs{k}{n} is Latin, then the set is a set of $k$ \emph{mutually
orthogonal Latin squares}, abbreviated \mols{k}{n}. Although it is traditional
to talk about \emph{sets} of MOFS or MOLS, in this paper we will often care
about the order of the elements within that set. Hence we will often think of
\emph{lists} of MOFS or MOLS.

It is well known that a \net{k}{n} is equivalent to a set of \molsb{k-2}{n}.
Under this correspondence, the lines of two arbitrary parallel classes of the
net correspond to the rows and the columns, while each of the remaining $k-2$
parallel classes correspond to the symbols in a different square. A
\emph{transversal} $t$ in a \net{k}{n} $N$ is a set of $n$ points in $N$ such
that $t$ is orthogonal to every line in $N$. In a related notion, a
\emph{transversal} in a Latin square is defined to be a set of $n$ cells that
have no rows nor columns in common, such that every symbol occurs exactly once
within the cells. For a set $\LL$ of \molsb{k-2}{n}, a \emph{common
transversal} is a set $t'$ of $n$ cells such that $t'$ is a transversal in each
of the $k-2$ Latin squares of $\LL$. A common transversal in a set of MOLS is
equivalent to a transversal in the corresponding net. A \net{k}{n} $N$ has $n$
disjoint transversals if and only if $N$ is a subnet of some \netb{k+1}{n}.

Let $\ztosublist{s}{t-1}$ be a list of positive integers, $\ztosublist{S}{t-1}$
be a list of multisets such that, for $0\le i<t$, the support of $S_i$ is
$\zto{s_i-1}$ and let $\lambda^i_k$ be the multiplicity of the symbol $k$ in
the multiset $S_i$. A \emph{mixed type orthogonal array} (\moa) with multisets
$\ztosublist{S}{t-1}$ is an $s\times t$ array in which the $i^\text{th}$ column
is a permutation of the multiset $S_i$, for $0\le i<t$. Also, for every pair of
disjoint columns $i$ and $j$ and symbols $k \in S_i$ and $\ell \in S_j$ there
should be exactly $\lambda^i_k\lambda^j_\ell/s$ rows where $k$ occurs in column
$i$ and $\ell$ occurs in column $j$.  Mixed type orthogonal arrays generalise
MOFS. From any \mofs{(t-2)}{n} $F_0,\dots,F_{t-3}$ we can build an $n^2\times
t$ \moa\ in which each row records the index $i$ of a row, the index $j$ of a
column and then the symbols $F_0[i,j],F_1[i,j],\dots,F_{t-3}[i,j]$. Moreover,
if a \moa\ has $s_0=s_1=n$ and $\lambda_k^0=\lambda_k^1=n$ for each symbol
$k\in\{0,\dots,n-1\}$ then the \moa\ can be formed from a set of MOFS in this
way. If $\lambda_k^i=n$ for all $i$ and $k$ then the \moa\ corresponds to a set
of MOLS.

We do not attach any significance to the order of the rows in a \moa, so we
consider two {\moa}s to be the same if one can be obtained from the other by
any permutation of the rows.  A \moa\ $O_1$ is a \emph{conjugate} of another
\moa\ $O_2$ if $O_2$ can be obtained by permuting the columns of $O_1$.  Two
{\moa}s $O_1$ and $O_2$ are \emph{isotopic} if $O_1$ can be obtained from $O_2$
by, for each $i$, applying some permutation of $\{0,\dots,s_i-1\}$ to the
symbols within column $i$.  Two {\moa}s $O_1$ and $O_2$ are \emph{paratopic},
or belong to the same \emph{species}, if $O_2$ is isotopic to a conjugate of
$O_1$. Any $s\times t$ {\moa} has $t!$ conjugates. In this paper we are only
interested in {\moa}s that represent MOFS. For that reason we will only
consider $n^2\times t$ {\moa}s with $t \geq 3$ and multisets $S_0$ and $S_1$
that have support $\zto{n-1}$, with multiplicity $n$ for each symbol in $S_0$
and $S_1$. The multisets $S_0$ and $S_1$ will represent the rows and columns of
the MOFS, respectively. Furthermore, for any {\moa} $O$, with multisets
$\otosublist{S}{t}$, we only consider conjugates of $O$ that respect the symbol
sets in the sense that the multisets for the conjugate also equal
$\otosublist{S}{t}$.

Let $F$ be a frequency square with symbol set of size $s$ and let $t\le s$. A
\emph{symbol homotopy} $\theta: \zto{s-1} \rightarrow \zto{t-1}$ is a mapping
from the symbol set of $F$ to a symbol set of size $t$ resulting in a frequency
square obtained by identifying symbols in $F$. Let $F'$ be a frequency square
with symbol set of size $t$. The frequency square $F$ is a \emph{refinement} of
$F'$ if a symbol homotopy $\theta$ exists such that $F' = \theta(F)$. A symbol
homotopy of a list of frequency squares $\otosublist{F}{k}$ is a list of
functions $(\theta_0,\dots,\theta_{k-1})$ such that $\theta_i$ is a symbol
homotopy of $F_i$, for $0\leq i<k$. A list of frequency squares
$\otosublist{G}{k}$ is a \emph{refinement} of $\otosublist{F}{k}$ if
$\otosublist{F}{k} = (\theta_0(G_0),\dots, \theta_{k-1}(G_{k-1}))$.

Let $\rel_1$ and $\rel_2$ be two relations on a \net{k}{n}, and suppose that
$\rel_2$ is a trivial relation. Then the symmetric difference of $\rel_1$ and
$\rel_2$ is a relation obtained by complementing the relational lines of
$\rel_1$ relative to each of the parallel classes that are included in
$\rel_2$. For example, by complementing two parallel classes of weight $2$ in a
relation of type $2^34^26$ we obtain a relation of type $24^268^2$.  An
\emph{odd relation} can be obtained from a relation by complementing a set of
relational lines contained in a single parallel class. Equivalently, an odd
relation is any set of lines $\relodd \subseteq L$ satisfying the equation
$\sum_{x\in \relodd} \alpha(x) = \bvec{1}$. Each point $p\in P$ in an odd
relation has weight $w(p) \equiv 1 \pmod{2}$.

\bigskip

The cardinality $N(n)$ of a largest set of MOLS$(n)$ has been of interest to
the mathematical community for centuries. Euler, while looking for new ways to
construct magic squares, famously conjectured that $N(4t-2)=1$ for all $t \geq
1$.  \citet*{tarry_probleme_1901} used exhaustive case analysis to prove
$N(6)=1$, confirming Euler's conjecture in that case.  However,
\citet*{bose_falsity_1959} discovered a counterexample to the conjecture of
order 22. Later the same year \citet*{parker_construction_1959} discovered a
counterexample of order 10. \citet*{bose_further_1960} provided a construction
that shows that $N(n)\ge2$ for all $n \geq 7$. \citet{egan_enumeration_2016}
counted and classified all sets of MOLS$(n)$ for $n\le9$. However, even the
value of $N(10)$ has proved elusive despite considerable effort (see
\citet{mckay_small_2007}). This context provides the motivation for the present
work.

It is easily seen that $N(n)\le n-1$.  A set of \molsb{n-1}{n} is said to be
\emph{complete}.  A complete set of MOLS$(n)$ exists if and only if there is a
projective plane of order $n$. As there is a Desarguesian plane for all prime
power orders, a complete set of MOLS exists for each prime power order.
\citet*{bruck_finite_1963} used geometric arguments to bound the size of the
largest set of MOLS when no complete set exists, showing that if $N(n) < n-1$
then $N(n) < n - 1 -(2n)^{1/4}$. This result was later improved by
\citet{metsch_improvement_1991}.

\citet*{stinson_short_1984} used the theory of transversal designs to give a
short proof of the non-existence of a pair of MOLS of order $6$. This was a
significant step given the exhaustive cases analysis required by
\citet*{tarry_probleme_1901}. Later, \citet*{dougherty_coding_1994-1} combined
net and coding theoretic techniques to give another short proof that $N(6)=1$.
\citet*{lam_non-existence_1989} proved the non-existence of a projective plane
of order $10$ using an extensive computational search. Thus there does not
exist a complete set of \mols{9}{10} and it follows from the theorems of
\citet{bruck_finite_1963} and \citet{metsch_improvement_1991} that $N(10)\le6$.
If all of the types of relations in Dukes and Howard's catalogue could be ruled
out, then Theorem~\ref{t:DH} would show that $N(10)\le3$. This paper takes a
first step in that direction, ruling out two of the types.

\section{Templates}\label{s:templates}

In this section we introduce a structure called a template, which will assist
in our study of relations.  Throughout this section $n$ will be an \emph{even}
positive integer.

A \emph{template} $T$ of order $n$ and type $(\lambda_0, \lambda_1,
\dots,\lambda_{k-1})$, with $\lambda_i$ even, is a list $(F_0, F_1, \dots,
F_{k-3})$ of \mofs{(k{-}2)}{n} with the following properties.
\begin{enumerate}
\item \label{def:tempproprowcol} $T$ has rows $(0,
  1,\dots,n-1)$ and columns $(0,
  1,\dots,n-1)$. The rows $(0, 1,\dots,
  \lambda_0-1)$ and the columns
  $(0,1,\dots,\lambda_1-1)$ are deemed to be relational.
  
\item \label{def:temppropfreq} For each $t \in \{2,\dots,k-1\}$ the
  binary frequency square $F_{t-2}$ has frequencies 
  $(n-\lambda_t, \lambda_t)$.

\item The type of a point $(i,j)$ is the $k$-tuple
  $\big(x,y, F_0[i,j], F_1[i,j],\dots,F_{k-3}[i,j]\big)$
  where $x=1$ if row $i$ is relational and 0 otherwise, and $y=1$
  if column $j$ is relational and 0 otherwise.
  A point $(i,j)$ is relational relative to a
  frequency square $F_\ell$ if $F_\ell[i,j] = 1$.

\item \label{def:tempparity}
  The weight of a point is the sum (in $\mathbb{Z}$) of the elements of its
  type $k$-tuple.  All points $p$ in $T$ have weight
  $w(p) \equiv \chi \pmod{2}$, where $\chi=(\sum_i \lambda_i)/2$.
  
\end{enumerate}

For simplicity we represent the type of a point in a template as a binary
string (i.e.~we omit the commas and parentheses).  Two templates are of
equivalent type if there is a permutation of one type that is equal to the
other. Two templates of equivalent type are \emph{isomorphic} if one template
can be obtained from the other by applying any sequence of the following
operations:
\begin{enumerate}
	\item Uniformly permuting the rows of the frequency squares in
          the template. Such a permutation must map relational rows to
          relational rows.
	\item Uniformly permuting the columns of the frequency squares
          in the template. Such a permutation must map relational
          columns to relational columns.
	\item Transposing all of the frequency squares in the template.
	\item Reordering the frequency squares in the template.
\end{enumerate}

A \net{k}{n} $N$ is said to be a \emph{refinement} of a template $T$ if the
list of \molsb{k-2}{n} that is equivalent to $N$ is a refinement of $T$.
Implicit in this statement is the need for us to have fixed the ordering of the
parallel classes of $N$, say by indexing them as $\Pi_0,\dots,\Pi_{k-1}$.
Under this ordering, the parallel class $\Pi_0$ corresponds to the rows in $T$,
the parallel class $\Pi_1$ corresponds to the columns of $T$ and each remaining
parallel class $\Pi_i$ corresponds to the frequency square $F_{i-2}$.  The
properties imposed on the frequency squares of the template guarantee that if a
net is a refinement of the template then the net satisfies a relation (or odd
relation) $\rel$ of type $(\lambda_0,\lambda_1,\dots,\lambda_{k-1})$.  Whether
$\rel$ is an odd relation depends on the parity of $\chi$.  Each relational
line in $\rel$ corresponds to a relational row, relational column, or a set of
$n$ relational points relative to a frequency square in the template. For every
net $N$ with a non-trivial relation, there exists a template such that $N$ is a
refinement of the template. This can be seen by considering the list of MOLS
equivalent to $N$ and applying a symbol homotopy that maps each symbol in each
square to 0 or 1 according to whether the line corresponding to that symbol is
non-relational or relational in $N$, respectively.  However, starting with a
template, it may or may not be possible to find a net that refines it. In
\sref{ss:temp4444} we will give a template that is not refined by any net.
However, we are only really interested in templates that have a refinement to a
net. So in the following sections we will often hypothesise the existence of
such a refinement when studying the structure of templates.

\subsection{Template Structures}\label{ss:tempstructure}

Let $\BB^k$ be the set of all length $k$ binary strings containing an even
number of ones. Let $\BB^k_{i,j}$ be the subset of $\BB^k$ such that $\bvec{b}
\in \BB^k_{i,j}$ if and only if both the $i^\text{th}$ and $j^\text{th}$ bits
of $\bvec{b}$ are $1$. Similarly, let $\bar{\BB}^k_{i,j}$ be the subset of
$\BB^k$ such that $\bvec{b} \in \bar{\BB}^k_{i,j}$ if and only if the
$i^\text{th}$ bit of $\bvec{b}$ is $1$ and the $j^\text{th}$ bit of $\bvec{b}$
is $0$. For a particular template $T$, let $t_\bvec{b}$ be the number of points
in $T$ of type $\bvec{b}$. The bitwise complement of a type $\bvec{b}$ is
denoted $\bar{\bvec{b}}$.

Let $N$ be a \net{k}{n} that admits a relation $\rel$ of type $(\lambda_0,
\lambda_1, \dots, \lambda_{k-1})$. By considering the orthogonality of lines in
$N$ we can derive a set of equations that the points of a template must satisfy
if the template is to be refined by $N$. For later simplicity, let $g_i =
\frac{n}{2}-\lambda_i$, for $0\leq i < k$. We have
\begin{align}
	\sum_{\bvec{b}\in \BB^k} t_\bvec{b}              & =
	n^2,\label{eqn:100points}
	\\
	\sum_{\bvec{b} \in \BB^k_{i,j}} t_\bvec{b}       & =
	\left(\frac{n}{2}-g_i\right)
	\left(\frac{n}{2}-g_j\right)\label{eqn:rellines}
	\quad
	\text{for all
		$i < j$,}
	\\
	\sum_{\bvec{b} \in \BB^k_{i,j}} t_{\bar{\bvec{b}}}    & =
	\left(\frac{n}{2}+g_i\right)
	\left(\frac{n}{2}+g_j\right)\label{eqn:nonrellines}
	\quad
	\text{for all
		$i < j$,}
	\\
	\sum_{\bvec{b} \in \bar{\BB}^k_{i,j}} t_\bvec{b} & =
	\left(\frac{n}{2}-g_i\right)\left(\frac{n}{2} +
	g_j\right)\label{eqn:relvnonrel} \quad
	\text{for all
		$i \not =j$.}
\end{align}

\Cref{eqn:100points} simply states that a net of order $n$ contains $n^2$
points. The set of equations $(\ref{eqn:rellines})$ are derived from the fact
that for every pair of parallel classes $\Pi_i$ and $\Pi_j$ in $N$, each
relational line from $\Pi_i$ shares exactly one point with each relational line
in $\Pi_j$. Similarly, the equations (\ref{eqn:nonrellines}) come from the fact
that every non-relational line in $\Pi_i$ shares exactly one point with each
non-relational line in $\Pi_j$. Finally the equations $(\ref{eqn:relvnonrel})$
are a consequence of every relational line in $\Pi_i$ sharing exactly one point
with each non-relational line in $\Pi_j$.

The equations
\labelcref{eqn:100points,eqn:rellines,eqn:nonrellines,eqn:relvnonrel} hold for
all $6$-nets of even order. Let $(\lambda_0,\lambda_1,\dots,\lambda_5)$ be the
type of a relation. The linear system consists of $61$ equations with a
variable for each of the $32$ point types. Unfortunately, the resulting system
is underdetermined, with $10$ degrees of freedom.  \Cref{table:solcounts} gives
the number of nonnegative integer solutions to the equations
\labelcref{eqn:100points,eqn:rellines,eqn:nonrellines,eqn:relvnonrel} for
certain relations on a \net{6}{10}.  \citet{dukes_group_2014} showed that every
relation on a \net{6}{10} that has $0 < \lambda_i < 10$ for $0 \le i \le 5$ is
equivalent to one of those listed in the table. The same paper demonstrates a
restriction which is not included in the system of equations
\labelcref{eqn:100points,eqn:rellines,eqn:nonrellines,eqn:relvnonrel}.  In
particular, they show that, for $i\in\{0,\dots,5\}$,
\begin{equation}\label{e:regcond}
  t_{111111} \leq \lambda_i\left\lfloor\frac{
    \sum_{j\ne i}\lambda_j-10}{4}\right\rfloor.
\end{equation}
An easy way to deduce \eref{e:regcond} is to consider a relational line $\ell$
in the $i$-th parallel class. All ten points in $\ell$ must meet an odd number
of the relational lines that are orthogonal to $\ell$. Allocating one such line
to each point, we then have $\sum_{j\ne i}\lambda_j-10$ lines as yet
unallocated, and each weight 6 point on $\ell$ uses four of them.  It follows
that \eref{e:regcond} holds, since we have $\lambda_i$ choices for $\ell$ and
every weight 6 point must lie on one of them.  In practice \eref{e:regcond}
only removes $16$ solutions for relations of type $2^44^2$ and $16$ solutions
for relations of type $4^6$. The only solution for relations of type $2^56$,
marked with a $\dagger$, satisfies \eref{e:regcond} but can be eliminated by
\cite[Prop.~2.9]{dukes_group_2014}.

\begin{table}[ht]
  \begin{center}
    \begin{tabular} {crr}
      \toprule
      Relation type & Solutions to \labelcref{eqn:100points,eqn:rellines,eqn:nonrellines,eqn:relvnonrel} & Solutions to \labelcref{eqn:100points,eqn:rellines,eqn:nonrellines,eqn:relvnonrel,e:regcond}
      \\
      \midrule
      $2^6$         & 1 & 1
      \\
      $2^56^1$      & 1\rlap{$^\dagger$} & 1\rlap{$^\dagger$}
      \\
      $2^44^2$      & 146 & 130
      \\
      $2^34^26^1$   & 1\,302 & 1\,302
      \\
      $2^24^4$      & 5\,286 & 5\,286
      \\
      $2^14^46^1$   & 61\,113 & 61\,113
      \\
      $4^6$         & 1\,832\,069 & 1\,832\,053
    \end{tabular}
    \caption{\label{table:solcounts}The number of solutions
      to the constraints on the $t_{\bvec{b}}$
	}
  \end{center}
\end{table}

Although our system of equations usually has many solutions, there is a
relationship between the counts of complementary point types:

\begin{theorem} \label{thm:complements}
	Let $n$ be even and let $N = (P, L)$ be a \net{6}{n} with relation
	$\rel$ of type $(\lambda_0, \lambda_1,\dots,\lambda_5)$. Then,
	\begin{align}
		t_{\bvec{b}} + t_{\bar{\bvec{b}}} = \frac{1}{16}n^2 +
		\frac{1}{4}\sum_{i < j}(-1)^{\bvec{b}_i + \bvec{b}_j}g_ig_j, \quad &
		\text{for all $\bvec{b}\in \BB^6$}.
		\label{eqn:allpoints}
	\end{align}
\end{theorem}

\begin{proof}
	Let $\bvec{z}$ be the point type of weight $0$ and $\bar{\bvec{z}}$ be
	its complement. \Cref{eqn:100points} contains every point type exactly
	once and every point type $\bvec{b} \in \BB^6 \setminus \{\bvec{z},
	\bar{\bvec{z}}\}$ is in exactly $8$ equations from
	\eref{eqn:relvnonrel}. It follows that the equation obtained from
	\eref{eqn:100points} by subtracting $1/8^\text{th}$ of each equation in
	\eref{eqn:relvnonrel} gives,
	\begin{align}\label{e:comprel}
		t_\bvec{z} + t_{\bar{\bvec{z}}} = \frac{1}{16}n^2 +
		\frac{1}{4}\sum_{i < j}g_ig_j.
	\end{align}

	Now consider a point type $\bvec{z}'$ of weight $2$ in $\rel$ with
	$\bvec{z}'_a = \bvec{z}'_b = 1$, where $a \neq b$.  Complementing the
	relational lines in $\Pi_a$ and $\Pi_b$ gives a new relation $\rel'$
	for which $g'_a = -g_a$ and $g'_b= -g_b$ and $g'_k = g_k$ for all
	$k\in\{0,\dots,5\}\setminus\{a,b\}$. Moreover, points of type
	$\bvec{z}'$ in $\rel$ have weight $0$ in $\rel'$ and vice versa.  Thus,
	there are $t_{\bvec{z}'}$ points of weight $0$ and
	$t_{\bar{\bvec{z}}'}$ points of weight $6$ in $\rel'$ (where each
	$t_{\bvec{b}}$ refers to counts in $\rel$, not in $\rel'$). So from
	(\ref{e:comprel}), we have
	\begin{align*}
	  t_{\bvec{z}'} + t_{\bar{\bvec{z}}'}
          = \frac{1}{16}n^2 + \frac{1}{4}\sum_{i < j}g'_ig'_j
          = \frac{1}{16}n^2 + \frac{1}{4}\sum_{i < j}(-1)^{\bvec{z}'_i + \bvec{z}'_j}g_ig_j.
	\end{align*}
	We have thus shown that (\ref{eqn:allpoints}) holds when $\bvec{b}$ has
	weight $0$ or $2$ in $\rel$. Moreover, the truth of
	(\ref{eqn:allpoints}) is clearly preserved if we replace $\bvec{b}$ by
	its complement. Since every point type in $\BB^6$ has weight $0$, $2$,
	$4$, or $6$, the result follows.
\end{proof}

A relation $\rel$ in a \net{6}{n} that ignores a parallel class, e.g.\ is of
type $(\lambda_0, \lambda_1, \dots, \lambda_4, 0)$, is a relation on a
\net{5}{n} that is a subnet of the \net{6}{n}. In some ways, every relation on
a \net{5}{n} can be thought of in this way, regardless of whether the
\net{5}{n} embeds in a \net{6}{n}. The reason is that the equations
\labelcref{eqn:100points,eqn:rellines,eqn:nonrellines,eqn:relvnonrel} still
hold for any \net{5}{n} if we invent a sixth parallel class that has zero
relational lines (it does not matter whether such a parallel class is actually
achievable). By taking $\lambda_5=0$ we ensure that there can be no points
incident with a relational line in the sixth parallel class. So for any
$\bvec{b}\in\BB^6$ one of $\bvec{b}$ or $\bar{\bvec{b}}$ will necessarily be
zero, and the other can be determined from \Cref{thm:complements}. In other
words, the number of points of each type is completely determined. This gives
the following equations for a \net{5}{n}:
\begin{align}\label{e:5neteq}
	t_{\bvec{b}} = \frac{1}{16}n^2 + \frac{1}{8}n\sum_i(-1)^{\bvec{b}_i}g_i + \frac{1}{4}
	\sum_{i < j}(-1)^{\bvec{b}_i + \bvec{b}_j}g_ig_j \quad &
	\text{for all $\bvec{b}\in \BB^5$}.
\end{align}
The same approach can be used when considering a \net{4}{n}. In this case we
can take $\lambda_4 = \lambda_5 = 0$, implying that there are no points
incident with a relational line in either $\Pi_4$ or $\Pi_5$. This gives the
following solution for the counts, by type, of all points in a \net{4}{n}:
\begin{align}\label{e:4neteq}
	t_{\bvec{b}} = \frac{1}{8}n^2 + \frac{1}{4}n\sum_i(-1)^{\bvec{b}_i}g_i +
	\frac{1}{4}\sum_{i < j}(-1)^{\bvec{b}_i + \bvec{b}_j}g_ig_j \quad &
	\text{for all $\bvec{b}\in \BB^4$}.
\end{align}

\citet{dukes_group_2014} eliminated the following types of relations on $4$
parallel classes: $2^4$, $2^36$, $2^24^2$ and $24^26$. In the first three of
these cases, (\ref{e:4neteq}) gives $t_{1111} < 0$, providing an alternative
way to see that such relations are impossible. Dukes and Howard also eliminated
a relation on $5$ parallel classes of type $2^44$. In this case
(\ref{e:5neteq}) with $\bvec{b} \in \{11101, 11011, 10111, 01111\}$ gives
$t_\bvec{b} < 0$.

\subsection{Templates for Relations of Type $4^4$}\label{ss:temp4444}

This subsection describes the structure of a template of type $4^4$ that has a
refinement to a $\net{4}{10}$.  The relational rows and relational columns of a
template impose a quadrant structure on the frequency squares in the template.
The $4$ subarrays are defined as follows:
\begin{itemize}
	\item $Q_1$ contains all points that are in relational rows and
	also in relational columns.
	\item $Q_2$ contains all points that are in
	relational rows and non-relational columns.
	\item $Q_3$ contains points that are in
	relational columns and non-relational rows.
	\item $Q_4$ contains points that are in non-relational rows and also in
	non-relational columns.
\end{itemize}

From \eref{e:4neteq} we know that $t_{0000} = 24$, $t_{1100} = t_{1010} =
t_{0110} = t_{1001} = t_{0101} = t_{0011} = 12$, and $t_{1111} = 4$.  As every
point in $Q_1$ is contained in a relational row and a relational column, $Q_1$
contains points of type $1100$ and $1111$ representing points of weight $2$ and
$4$. Similarly, $Q_2$ contains points of type $1001$ and $1010$ and $Q_3$
contains points of type $0101$ and $0110$ representing the points of weight $2$
contained in either a relational row or a relational column but not both.
Lastly, the points in $Q_4$ are in non-relational rows and columns and thus
have weight $0$ or $2$ and type $0000$ or $0011$. In the following lemmas, we
will use the symbol $*$ as a wildcard which could represent either a $1$ or $0$
as required.

\begin{lemma}\label{lem:4444template}
   Let $N$ be a \net{4}{10} that is a refinement of a template $T$ of type
   $4^4$.
  \begin{enumerate}[label={\rm(\roman*)}]
  \item The four points of weight $4$ contained in template $T$ do
    not share any row or column.
  \item Each row in $Q_2$ and column in $Q_3$ contains three points of
    type ${*}{*}10$ and three points of type ${*}{*}01$.
  \item Each relational line in $\Pi_2\cup\Pi_3$ is incident to
    exactly one point in $Q_1$ and three points in each of $Q_2$, $Q_3$,
    and $Q_4$.
  \item Each column in $Q_2$ and row in $Q_3$ contains two points of
    type ${*}{*}01$ and two points of type ${*}{*}10$. Moreover, each
    row and column in $Q_4$ contains two points of type $0011$ and
    four points of type $0000$.
  \end{enumerate}
\end{lemma}

\begin{proof}
 	Consider a relational line $\ell$ with respect to a relation of type $4^4$
	 in a \net{4}{10}. The ten points on $\ell$ each have weight 2 or 4, and
	the total of their weights is $10+3\times4=22$. Hence there must be one
	weight 4 point and nine weight 2 points on $\ell$. Parts (i), (ii) and
	(iii) of the Lemma now follow easily.

	For part (iv) we consider instead a non-relational line $\ell'$. Its points
	all have weight 0 or 2, with a total weight of $3\times4=12$. Hence it
	has six points of weight 2 and four points of weight 0. Each parallel
	class orthogonal to $\ell'$ contains six non-relational lines. Four of
	these six lines hit points of weight zero in $\ell'$ and the other two
	must hit points of weight 2. Part (iv) follows.  
\end{proof}

\begin{figure}[!ht]
	\begin{center}
		\begin{tikzpicture}
			\begin{scope}
				\matrix (A) [square freq small] {
					11 & 00 & 00 & 00 & 01 & 01 &
					01 & 10 &
					10 & 10
					\\
					00 & 11 & 00 & 00 & 01 & 01 &
					01 & 10 &
					10 & 10
					\\
					00 & 00 & 11 & 00 & 10 & 10 &
					10 & 01 &
					01 & 01
					\\
					00 & 00 & 00 & 11 & 10 & 10 &
					10 & 01 &
					01 & 01
					\\
					10 & 10 & 01 & 01 & 11 & 11 &
					00 & 00 &
					00 & 00
					\\
					10 & 10 & 01 & 01 & 11 & 11 &
					00 & 00 &
					00 & 00
					\\
					10 & 10 & 01 & 01 & 00 & 00 &
					11 & 11 &
					00 & 00
					\\
					01 & 01 & 10 & 10 & 00 & 00 &
					11 & 11 &
					00 & 00
					\\
					01 & 01 & 10 & 10 & 00 & 00 &
					00 & 00 &
					11 & 11
					\\
					01 & 01 & 10 & 10 & 00 & 00 &
					00 & 00 &
					11 & 11
					\\
				};
				\draw[ultra thick, red] (A-1-1.north
				west)
				rectangle (A-10-4.south east);
				\draw[ultra thick, red] (A-1-1.north
				west)
				rectangle (A-4-10.south east);
				\draw[ultra thick] (A-1-1.north west)
				rectangle
				(A-10-10.south east);
			\end{scope}
		\end{tikzpicture}
	\caption{\label{fig:44symmetric}A template of order $10$, type $4^4$.}
	\end{center}
\end{figure}

Exhaustive computations found that there are $6\,965$ isomorphism classes of
templates of type $4^4$ satisfying conditions (i), (ii) and (iv) of
\Cref{lem:4444template}. See \Cref{fig:44symmetric} for one example.  We stress
that to generate this catalogue we did not impose an assumption that templates
must have a refinement to a net. Indeed, our next result will show that some of
the $6\,965$ templates cannot be refined to a net. 

\begin{lemma}\label{lem:nonexistence}
  Let $T$ be the template given in \fref{fig:44symmetric}.  There does not
  exist a net $N$ that is a refinement of~$T$.
\end{lemma}

\begin{proof}
	Suppose for the sake of contradiction that there exists a \net{4}{10}
	$N$ that is a refinement of $T$. Let $\ell \in \Pi_2$ be a relational
	line in $N$ incident to the point $(0,0)$. By
	\Cref{lem:4444template}(iii), $\ell$ is incident to no other point in
	$Q_1$.  It follows that $\ell$ is incident with one point in rows
	$4$--$6$ of column $1$. This implies that $\ell$ is incident with two
	points in rows $4$--$6$ in $Q_4$.  Similarly, $\ell$ is incident with
	two points in columns $4$--$6$ in $Q_2$ and thus columns $4$--$6$
	contain one point incident to $\ell$ in $Q_4$. Thus point $(6,7)$ is
	incident to $\ell$.

	Now let $\ell' \in \Pi_2$ be a relational line such that $\ell$ is
	incident to the point $(1,1)$. By the same logic used above, $\ell'$ is
	incident to point $(6,7)$. This is a contradiction as both $\ell$ and
	$\ell'$ are in the same parallel class and $\ell \not = \ell'$. It
	follows that there does not exist a net $N$ that is a refinement of
	$T$.
\end{proof}

The key point, guaranteed by \lref{lem:4444template}, is that any \net{4}{10}
that satisfies a relation of type $4^4$ is isomorphic to a net that is a
refinement of one of our $6\,965$ templates. The computations that we report in
\sref{section:nets4444} will rely on this fact.

\subsection{Templates for Odd Relations of Type $4^22^3$}

In this subsection we consider templates for odd relations of type $4^22^3$.  A
significant difference between our task here and the one in the previous
subsection is that the parallel classes now do not all have the same weight.
Therefore, the ordering of the parallel classes makes a material difference. In
a template the rows and columns play a different role to the other parallel
classes as represented by the frequency squares. Hence, we need to consider
three equivalent types, namely $2^34^2$, $242^24$ and $4^22^3$.  A simple
recursive backtracking algorithm, with isomorphism rejection, can be used to
find the set of templates of each type. The sets of templates generated for
$2^34^2$ and $242^24$ were prohibitively large, with the enumeration of
templates of type $2^34^2$ not completing after producing billions of
templates.  However, we found that there are only $30$ isomorphism classes of
odd templates of type $4^22^3$. See \Cref{fig:44222template} for one example.

  \begin{figure}[!ht]
	\begin{center}
		\begin{tikzpicture}
			\begin{scope}
				\matrix (A) [square freq small] {
		111 & 100 & 010 & 001 & 000 & 000 & 000 & 000 & 000 & 000 \\
		100 & 111 & 001 & 010 & 000 & 000 & 000 & 000 & 000 & 000 \\
		010 & 001 & 111 & 100 & 000 & 000 & 000 & 000 & 000 & 000 \\
		001 & 010 & 100 & 111 & 000 & 000 & 000 & 000 & 000 & 000 \\
		000 & 000 & 000 & 000 & 100 & 100 & 010 & 010 & 001 & 001 \\
		000 & 000 & 000 & 000 & 100 & 100 & 010 & 001 & 010 & 001 \\
		000 & 000 & 000 & 000 & 010 & 010 & 100 & 001 & 001 & 100 \\
		000 & 000 & 000 & 000 & 010 & 001 & 001 & 100 & 100 & 010 \\
		000 & 000 & 000 & 000 & 001 & 010 & 001 & 100 & 100 & 010 \\
		000 & 000 & 000 & 000 & 001 & 001 & 100 & 010 & 010 & 100 \\
				};
				\draw[ultra thick, red] (A-1-1.north
				west)
				rectangle (A-10-4.south east);
				\draw[ultra thick, red] (A-1-1.north
				west)
				rectangle (A-4-10.south east);
				\draw[ultra thick] (A-1-1.north west)
				rectangle
				(A-10-10.south east);
			\end{scope}
		\end{tikzpicture}
	\caption{
	  \label{fig:44222template}A template of order 10, type
          $4^22^3$ corresponding to $\dagger$ in
	  \Cref{table:templatecompletions}.}
	\end{center}
\end{figure}

The following lemma describes the structure of a template $T$ of type $4^22^3$,
and assisted in the generation of the 30 templates.  The equation
\eref{e:5neteq} for the number of points with a given type in a relation can be
converted into an equation for the number of points of a given type in an odd
relation by flipping the rightmost bit in the point types and exchanging all
instances of $g_5$ with $-g_5$. Doing so, we find that $t_{10000}=t_{01000} =
24$, $t_{00100} = t_{00010} = t_{00001} = 12$, and $t_{11100}= t_{11010} =
t_{11001} = t_{11111} = 4$ for an odd relation of type $4^22^3$. Using similar
notation as in the previous subsection we then have:

\begin{lemma}\label{lem:44222template}Consider a template of order $10$,
  type $4^22^3$.
  \begin{enumerate}[label={\rm(\roman*)}]
  \item The points in quadrants $Q_2$ and $Q_3$ are of type
    ${*}{*}000$.
  \item Every point in $Q_1$ is of type $11001$, $11010$, $11100$, or
    $11111$ with each of these types contained in each row and column once.
  \item Every point in $Q_4$ is of type $00001$, $00010$, or $00100$
    with each of these types appearing twice in each row and column.
  \end{enumerate}
\end{lemma}

\begin{proof}
  Part (i) follows immediately from the counts of each point type.
  
  Next, consider the four points within a row or column of $Q_1$. These
  four points have total weight $4+4+3\times2=14$. Hence they must
  consist of one point of weight 5 and three points of weight 3. Also,
  the four points must contain two which are relational for $\Pi_i$,
  for $2\le i\le 4$. Part (ii) follows.

  Part (iii) follows from the requirement for each non-relational row
  or column to meet each relational line in $\Pi_2$, $\Pi_3$ and $\Pi_4$.
\end{proof}

\Cref{table:templatecompletions} lists all $30$ templates of type $4^22^3$.
Each template can be obtained by replacing quadrant $Q_4$ of
\Cref{fig:44222template} with the $6\times 6$ subarray encoded in the
appropriate column of \Cref{table:templatecompletions}. To decode, interpret
each six digit number in the table as one row of the subarray and replace $1$
with $100$, $2$ with $010$ and $3$ with $001$.  Other features of
\Cref{table:templatecompletions} will be explained when we consider it further,
in \Cref{s:44222sym}.

\begin{table}[htb]
    \centering
        \begin{tabular}{cr@{\hspace*{25pt}}lr}
            \toprule
            \multicolumn{1}{l}{Symmetry} & \multicolumn{1}{l}{} & \multicolumn{1}{l}{}             \\
            \multicolumn{1}{l}{type} & $|\aut{T}|$ & \multicolumn{1}{c}{Encoding of $Q_4$ \ } & \multicolumn{1}{c}{$|\Omega'(T)|$} \\
            \midrule
            $S_3$ 	& 24 		& 112233 121323 223131 233112 312321 331212 			& 1\,392 	\\
            $S_3$ 	& 48 		& 112233 112323 231132 231312 323121 323211  			& 5\,967 	\\
            $S_3$ 	& 96 		& 112233 121323 213312 233121 321231 332112 			& 361 		\\
            $S_3$ 	& 96 		& 112233 112233 223311 231312 323121 331122 			& 2\,932 	\\
            $S_3$ 	& 144 		& 112233 121323 223311 232131 313212 331122  			& 262 		\\
            $S_3$ 	& 192 		& 112233 112233 231312 233121 321321 323112 			& 1\,513 	\\
            $S_3$ 	& 288 		& 112233 122313 223311 233121 311232 331122  			& 143	 	\\
            $S_3$ 	& 1\,152 	& 112233 112233 231312 231312 323121 323121  			& 1\,887 	\\
            $S_3$ 	& 9\,216 	& 112233 112233 223311 223311 331122 331122  			& 342 		\\
            $C_2$ 	& 8 		& 112233 232113 311232 321321 233121 123312  			& 2\,943	\\
            $C_2$ 	& 8 		& 112233 112323 231132 233211 321321 323112  			& 5\,924 	\\
            $C_2$ 	& 8 		& 112233 112323 223131 231312 323211 331122  			& 6\,027	\\
            $C_2$ 	& 8 		& 112233 132213 221331 231312 323121 313122  			& 11\,654	\\
            $C_2$ 	& 16 		& 112233 232113 311232 321312 233121 123321  			& 1\,460	\\
            $C_2$ 	& 16 		& 112233 113322 231123 232311 321231 323112  			& 2\,863 	\\
            $C_2$ 	& 16 		& 112233 112323 231132 233121 321312 323211  			& 2\,971 	\\
            $C_2$ 	& 16 		& 112233 112323 231132 233211 321312 323121  			& 3\,037 	\\
            $C_2$ 	& 32 		& 112233 112323 231231 233112 321321 323112  			& 2\,842 	\\
            $C_2$ 	& 32 		& 112233 121323 212331 233211 323112 331122  			& 2\,986	\\
            $C_2$ 	& 48 		& 112233 112323 223131 231312 323121 331212  			& 1\,961 	\\
            $C_2$ 	& 64 		& 112233 113322 231213 232131 321321 323112  			& 2\,821	\\
            $C_2$ 	& 64 		& 112233 112233 231312 231321 323112 323121$^*$ 		& 6\,018 	\\
            $C_2$ 	& 96 		& 112233 232113 231231 321312 313122 123321  			& 926		\\
            $C_2$ 	& 128 		& 112233 112323 221331 233112 323112 331221$^\dagger$  	& 3\,048	\\
            $C_2$ 	& 192 		& 112233 113322 231213 233121 321312 322131  			& 507		\\
            $C_2$ 	& 192 		& 112233 113322 231213 232131 321312 323121  			& 1\,898	\\
            -- 		& 8 		& 112233 121323 213132 231321 323211 332112  			& 2\,803	\\
            -- 		& 8 		& 112233 121323 212331 233112 323211 331122  			& 5\,649	\\
            -- 		& 8 		& 112233 112323 221331 233112 323121 331212  			& 11\,822	\\
            -- 		& 16 		& 112233 113322 231123 231231 322311 323112  			& 5\,867	\\
                                  
            \end{tabular}

    \caption{\label{table:templatecompletions}The $30$ odd
          templates of type $4^22^3$.}
\end{table}

\section{Refinements of Templates}
\label{s:refining}

In this section we will discuss the computational techniques used to find all
possible refinements of a template $T$ to a net $N$. We will first formalise
the notion of a partial \net{k}{n}, which allows for a stepwise refinement by
adding one line at a time.

A \emph{partial $k$-net} $\PP = (P,L)$ of order $n$ is an incidence structure
consisting of a set $P$ of $n^2$ points and a set $L$ of at most $kn$ lines
satisfying the following conditions.

\begin{enumerate}
	\item Every line is incident to $n$ points.
	\item Every point is incident to at most $k$ lines.
	\item Lines are either parallel or orthogonal.
	\item The set of lines partition into $k$ (possibly empty) sets of parallel
	lines. 
\end{enumerate}

It is easy to see that partial nets generalise nets as each set of parallel
lines contains at most $n$ lines. For a given partial net $\PP = (P,L)$, a set
of parallel lines $\Pi \subseteq L$ is a partial parallel class if $|\Pi| < n$,
otherwise the set is a parallel class. Let $\LL^\Pi$ denote the set of lines in
$\binom{P}{n}$ not contained in $\Pi$ that are both orthogonal to all lines in
$L\setminus\Pi$ and parallel to every line in $\Pi$. For any line $\ell \in
\LL^\Pi$ we can construct a new partial net $\PP' = (P, L\cup\{\ell\})$. Any
partial net $\PP' = (P, L')$ with $L \subset L'$ is an \emph{extension} of
$\PP$.

As a partial net can have empty partial parallel classes, a partial $k$-net is
also a partial $(k+1)$-net, where the last partial parallel class is empty. A
partial net is trivial if the partial net has exactly two parallel classes and
every partial parallel class is empty. Up to isomorphism, the $n\times n$ grid
is the only trivial partial net of order $n$. For $k\geq3$, every \net{k}{n} is
an extension of the trivial partial $k$-net of order $n$.

We now describe a simple exhaustive backtracking algorithm for generating all
possible $k$-nets of order $n$ using lines from within some predetermined set
of lines.  Starting with the trivial partial \net{k}{n} $N$, for each partial
parallel class $\Pi_i$ we first determine a set $\LL_{0,i} \subseteq
\LL^{\Pi_i}$.  Elements of $\bigcup_j\LL_{0,j}$ will be referred to as
\emph{candidate lines}. We will say more about the choice of candidate lines
below.  Let $\LL_0=(\LL_{0,2},\dots,\LL_{0,k-1})$.  Let $S$ contain a list of
indices of the partial parallel classes.  Each index should be repeated in $S$
as many times as the number of lines we wish to add to the partial parallel
class of that index.  The list will be used to direct the search and will be
referred to as our \emph{search strategy}.  The ordering of the elements within
the search strategy has algorithmic implications that will be discussed below.

\Cref{alg:backtracking} proceeds via recursion. At each stage $i$, we have a
partial net $\PP_i = (P, L_i)$ and a tuple $\LL_i =
(\LL_{i,2},\dots,\LL_{i,k-1})$ of sets of lines available for forming
extensions of $\PP_i$.  In stage $i$ we select a new line from $\LL_{i,S[i]}$
and add it to the partial parallel class $\Pi_{S[i]}$.  For example, if $S =
[2,3,2,3,2,3,\dots,2,3]$ then we would seek a \net{4}{n} by alternately
extending the two partial parallel classes $\Pi_2$ and $\Pi_3$.  We are only
interested in partial nets that might extend to a \net{k}{n}, and in such a net
every point must be incident with a line in $\Pi_{S[i]}$. So, our recursive
step need only consider the set of lines that are incident to a single point
$p$ that is not yet contained in a line of the partial parallel class
$\Pi_{S[i]}$. To limit the branching factor of our recursive step we choose $p$
such that the number of lines incident to $p$ in $\LL_{i,S[i]}$ is a minimum.
The same heuristic was used by \citet{mathon_searching_1997} and
\citet{best_thesis_2018}, and easily justifies the time required to identify
$p$ in terms of the speed-up it achieves.

Let $j=S[i]$.  If $p$ is not contained in any line of $\LL_{i,j}$, then the
point $p$ cannot be covered by any line and thus the partial parallel class
$\Pi_j$ cannot be completed. Otherwise, for each line $\ell \in \LL_{i,j}$ that
is incident to $p$, we obtain a partial net $\PP_{i+1} = (P, L_i \cup
\{\ell\})$.  Then, we obtain the set $\LL_{i+1,j} \subset \LL_{i,j}$ by taking
every line in $\LL_{i,j}$ that is disjoint from $\ell$. Also, for every partial
parallel class $\Pi_m$ with $m \not = j$, we obtain the set $\LL_{i+1,m}
\subseteq \LL_{i,m}$ by taking every line in $\LL_{i,m}$ that is orthogonal to
$\ell$. This gives the tuple $\LL_{i+1} = (\LL_{i+1,2}, \dots, \LL_{i+1,k-1})$.
A recursive call on $\PP_{i+1}$ and $\LL_{i+1}$ is then used to determine if
further extensions of the partial net $\PP_{i+1}$ exist. If every element of
$S$ has been exhausted then we have reached the desired partial net and the
result is recorded.

\begin{algorithm}[h]
  \caption{\label{alg:backtracking}Exhaustively search for partial nets within a given family of lines.} 
  \begin{algorithmic}
    \Function{Add\_Line}{$i$, $\PP_i = (P, L_i)$, $\LL_i$, $S$}
    \If{$i \geq |S|$}
    \State Record $\PP$
    \Else
    \State Let $\Pi$ be the partial parallel class indexed by $S[i]$
    \State Choose $p\in P$ not incident to
    any line in $\Pi$ and incident to the fewest lines in $\LL_{i,S[i]}$
    \For{$\ell \in \LL_{i,S[i]}$ incident to $p$}
    \State Let $\LL_{i+1,S[i]}$ be the set of lines in $\LL_{i,S[i]}$
    that are disjoint from $\ell$
    \ForAll{$j \not = S[i]$}
    \State Let $\LL_{i+1,j}$ be the set of lines in
    $\LL_{i,j}$ that are orthogonal to $\ell$
    \EndFor
    \State Let $\LL_{i+1} = (\LL_{i+1,2},
    \dots, \LL_{i+1,k-1})$
    \State \Call{Add\_Line}{$i+1$, $(P, L_i\cup \{\ell\}),
      \LL_{i+1}$, $S$}
    \EndFor
    \EndIf
    \EndFunction
  \end{algorithmic}
\end{algorithm}

The algorithm as described thus far is quite general. For our specific purposes
we will want to restrict the class of partial nets it generates by choosing
candidate lines that respect a given template $T$ in a sense we now describe.
If a \net{k}{n} $N$ refines $T$, then for each frequency square $F$ in $T$ the
corresponding lines in $N$ hit points that have a constant value in $F$.  Thus,
for each given parallel class $\Pi_i$ and its assigned frequency square
$F_{i-2}$, we restrict $\LL_{0,i}$ to be the set of lines in $\LL^{\Pi_i}$ that
contain only relational points or contain only non-relational points, relative
to $F_{i-2}$. It follows that any net resulting from running
\Cref{alg:backtracking} with these candidate lines must satisfy the relation
indicated by the template $T$. Similarly, if no net is returned from
\Cref{alg:backtracking} then there does not exist any net $N$ that refines $T$.

\subsection{Symmetry Breaking}
\label{section:sym}

\Cref{alg:backtracking} is sufficient for enumerating all \mols{2}{10} that
satisfy a relation, but it is too slow for dealing with relations on
\mols{3}{10}.  In this section we outline an improved algorithm, which takes
advantage of symmetries of a template $T$ to reduce the computation time.  Let
$\aut{T}$ be the group of all automorphisms of $T$ induced by the template
isomorphism relation. Suppose $\LL_0 = (\LL_{0,2}, \dots,\LL_{0,k-1})$ and that
the candidate lines in $\LL = \bigcup_i\LL_{0,i}$ have been selected with
respect to $T$.  Consider the induced action of $\aut{T}$ on $\LL$.  For a line
$\ell \in\LL_{0,2}$ we let $\OO_{\ell}=\{\ell'\in \LL: \phi(\ell) = \ell'
\text{ for some } \phi \in \aut{T}\}$ denote the orbit of $\ell$ under this
action. As template isomorphism allows for the reordering of frequency squares,
the orbit $\OO_\ell$ may or may not contain lines from $\LL_{0,j}$ where $j>2$.
Let $\ell^*$ denote the lexicographically least element of
$\LL_{0,2}\cap\OO_{\ell}$, which we call the \emph{orbit leader}.  Suppose that
a net $N$ is a refinement of $T$.  By definition, if $N$ includes $\ell$ then
$N$ is isomorphic to a refinement of $T$ that includes $\ell^*$.  We therefore
lose no generality in assuming that $N$ includes $\ell^*$.  Moreover, we can
impose an arbitrary ordering on $\OO=\{\OO_\ell:\text{non-relational
}\ell\in\LL_{0,2}\}$ and then assume that $\ell$ was chosen from the first
orbit in $\OO$ that contains any non-relational line in $N$.
\Cref{alg:symmetrybroken} enumerates partial nets that might extend to a net
$N$ satisfying these additional assumptions. In particular, for each orbit in
$\OO_\ell\in\OO$ we do one computation using \Cref{alg:backtracking} where we
first discard all orbits that are earlier in the ordering of $\OO$ and also
assume that our partial net includes the orbit leader $\ell^*$ from $\OO_\ell$.
This produces a dramatic speed-up compared to using \Cref{alg:backtracking} on
its own, especially for orbits which occur late in the ordering of $\OO$.  The
search strategy as specified by the list $S$, is chosen from a set of search
strategies $\MS$ before invoking \Cref{alg:backtracking}. This choice, which
affects efficiency but not correctness, will be described in more detail in
\Cref{s:44222}.  For the time being, we merely remark that there are reasons to
choose different strategies at different stages of the computation.

\begin{algorithm}[h]
  \caption{\label{alg:symmetrybroken}Symmetry breaking} 
  \begin{algorithmic}
    \Function{Sym\_Break}{$\PP = (P, L)$, $\LL_0$, $\OO$, $\MS$}
    \ForAll{$\OO_i \in \OO$}
      \State Let $\ell^*$ be the orbit leader of $\OO_i$
      \State Let $\PP_0 = (P, L \cup \{\ell^*\})$
      \For{$k \geq 2$}
      \State $\LL_{0,k}' = \LL_{0,k} \setminus{\bigcup_{j<i}} \OO_j$
      \EndFor
      \State Let $\LL_{0,2}^*$ be the set of lines in $\LL_{0,2}'$
      that are disjoint from $\ell^*$
      \ForAll{$k > 2$}
      \State Let $\LL_{0,k}^*$ be the set of lines in
      $\LL_{0,k}'$ that are orthogonal to $\ell^*$
      \EndFor
    \State $\LL_0^* = (\LL_{0,2}^*,\dots,\LL_{0,k-1}^*)$
    \State Select appropriate $S\in \MS$
    \State \Call{Add\_Line}{$0$, $\PP_0$, $\LL_0^*$, $S$}
    \EndFor
    \EndFunction
  \end{algorithmic}
\end{algorithm}

\section{Nets with Relations of Type $4^4$}\label{section:nets4444}

By construction, any \net{4}{10} that satisfies a relation of type $4^4$ must
be isomorphic to one that is a refinement of one of the $6\,965$ templates
constructed in \sref{ss:temp4444}.

For each template $T$, the set of candidate lines was computed using a simple
recursive backtracking algorithm. This computation produced approximately
$14\,000$ lines for each of the two frequency squares in $T$.
\Cref{alg:backtracking} was then used to search for extensions of the trivial
partial \net{4}{10}. Finally, the output of \Cref{alg:backtracking} was
screened for isomorphs.  In this way we obtained a set $\Omega$ of species
representatives of all \mols{2}{10} satisfying a relation of type $4^4$. The
search strategy $S$ was chosen such that \Cref{alg:backtracking} alternated
between the partial parallel classes in the partial net. This choice of $S$
proved much faster than attempting to complete the partial parallel classes one
at a time. All templates that had a refinement into some net had refinements to
at least $31$ non-isomorphic nets and at most $46\,161$ non-isomorphic nets.
Of the 6965 templates, only ten had no refinement to a net, including the
template discussed in \Cref{lem:nonexistence}. By enumerating $\Omega$, we
found:

\begin{lemma}
	There are $18\,526\,320$ isomorphism classes of $4$-nets of
        order $10$ satisfying at least one non-trivial relation.
\end{lemma}

The dimension of each pair of MOLS in $\Omega$ was computed. These results
confirmed the results of \citet{delisle_masters_2006} who, under the
supervision of Wendy Myrvold, computed the set of all \mols{2}{10} satisfying
at least two linearly independent non-trivial relations (dimension at most
$35$). The number of \mols{2}{10} with dimension at most 36 is listed in
\Cref{table:44dim}.

\begin{table}[h]
  \begin{center}
    \begin{tabular}{cr}
      \toprule
      Dimension & Species    \\
      \midrule
      34        & 6        \\
      35        & 85       \\
      36        & 18\,526\,229 \\
      \bottomrule
    \end{tabular}
    \caption{\label{table:44dim}The pairs of MOLS in $\Omega$ classified by
	dimension.}
  \end{center}
\end{table}

For each pair of MOLS in $\Omega$ we computed the set of common transversals.
No pair contained more than $5$ common disjoint transversals.
\Cref{table:44dim} shows the number of pairs of MOLS generated, partitioned
according to the number of disjoint common transversals.

\begin{table}[htb!]
	\centering
	
	\small{
		\begin{tabular}{ccr}
			\toprule
			\multicolumn{1}{l}{Disjoint}     & \multicolumn{1}{l}{}             & \multicolumn{1}{l}{}       \\
			\multicolumn{1}{l}{Transversals} & \multicolumn{1}{l}{Transversals} & \multicolumn{1}{c}{Species}  \\
			\midrule
			0                                & @0                                & 14\,389\,542                   \\
			1                                & @1                                & 3\,634\,655                    \\
			1                                & @2                                & 274\,766                     \\
			1                                & @3                                & 8\,601                       \\
			1                                & @4                                & 164                        \\
			1                                & @5                                & 3                          \\
			1                                & @6                                & 6                          \\
			1                                & @7                                & 8                          \\
			1                                & @8                                & 12                         \\
			1                                & @9                                & 1                          \\
			1                                & 10                               & 8                          \\
			1                                & 11                               & 1                          \\
			1                                & 12                               & 5                          \\
			1                                & 13                               & 1                          \\
			1                                & 16                               & 1                          \\
			2                                & @2                                & 185\,303                     \\
			2                                & @3                                & 28\,395                      \\
			2                                & @4                                & 1\,845                       \\
			2                                & @5                                & 73                         \\
			2                                & @6                                & 11                         \\
			2                                & @7                                & 5                          \\
			2                                & @8                                & 6                          \\
			2                                & @9                                & 2                          \\
			3                                & @3                                & 2\,339                       \\
			3                                & @5                                & 45                         \\
			3                                & @4                                & 466                        \\
			3                                & @6                                & 15                         \\
			3                                & @7                                & 6                          \\
			3                                & @8                                & 6                          \\
			3                                & 10                               & 1                          \\
			3                                & 12                               & 1                          \\
			3                                & 17                               & 1                          \\
			3                                & 18                               & 1                          \\
			3                                & 21                               & 4                          \\
			4                                & @4                                & 11                         \\
			4                                & @5                                & 3                          \\
			4                                & @7                                & 1                          \\
			4                                & @8                                & 2                          \\
			4                                & 10                               & 2                          \\
			5                                & @7                                & 1                          \\
			5                                & 19                               & 1                         
			\end{tabular}

	}

	\caption{\label{table:4444transversals}The MOLS
          in $\Omega$ classified by number of
          transversals and disjoint transversals.}
\end{table}

As no pair of MOLS in $\Omega$ was extendible to a triple, we have the
following theorem.

\begin{theorem}\label{thm:main_4444}
	Let $N$ be a \net{4}{10} that is a subnet of a
        \net{5}{10}. Then $N$ does not satisfy a non-trivial relation.
\end{theorem}

Inferring that we achieve equality in \eref{e:dimNbound} in this situation, we
have:

\begin{corollary}
  Let $N$ be a \net{4}{10} that is a subnet of a \net{5}{10}.
  Then $\dim(N) = 37$.
\end{corollary}

Two species in $\Omega$ have $5$ disjoint common transversals, see
\Cref{fig:4444with5trans1,fig:4444with5trans2}. \citet*{delisle_masters_2006}
found $4$ species of \mols{2}{10} with $3$ disjoint common transversals.
\citet*{brown_pair_1993} found a species of \mols{2}{10} with $4$ disjoint
common transversals. Only one species of \mols{2}{10} with more than 4 disjoint
common transversals was previously known, namely the pair with $7$ common
disjoint transversals found by \citet{egan_enumeration_2016}.

\begin{figure}
	\begin{center}
		\begin{tikzpicture}
			\begin{scope}
				\matrix (A) [square freq small] {
					0 & 1 &
					2 &
					|[fill=t1]|3 & |[fill=t2]|4 &
					|[fill=t3]|5 &
					|[fill=t4]|6 & 7 &
					|[fill=t5]|8 & 9 \\
					3 & |[fill=t3]|6 &
					|[fill=t1]|9 &
					7 & |[fill=t5]|0 &
					4 &
					1 & |[fill=t2]|5 &
					2 & |[fill=t4]|8 \\
					|[fill=t2]|1 & |[fill=t1]|2 &
					|[fill=t4]|5 &
					8 & 7 &
					9 &
					|[fill=t5]|4 & 3 &
					6 & |[fill=t3]|0 \\
					|[fill=t1]|7 & |[fill=t2]|9 &
					8 &
					4 & |[fill=t3]|1 &
					3 &
					2 & |[fill=t4]|0 &
					5 & |[fill=t5]|6 \\
					|[fill=t4]|4 & 5 &
					|[fill=t5]|1 &
					|[fill=t2]|2 & |[fill=t1]|6 &
					0 &
					9 & 8 &
					|[fill=t3]|7 & 3 \\
					6 & 4 &
					7 &
					|[fill=t4]|1 & 5 &
					|[fill=t1]|8 &
					|[fill=t3]|3 & |[fill=t5]|9 &
					|[fill=t2]|0 & 2 \\
					5 & |[fill=t5]|3 &
					6 &
					|[fill=t3]|9 & 8 &
					|[fill=t4]|2 &
					|[fill=t1]|0 & 4 &
					1 & |[fill=t2]|7 \\
					|[fill=t3]|8 & 0 &
					|[fill=t2]|3 &
					6 & 2 &
					|[fill=t5]|7 &
					5 & |[fill=t1]|1 &
					|[fill=t4]|9 & 4 \\
					9 & 8 &
					0 &
					|[fill=t5]|5 & |[fill=t4]|3 &
					|[fill=t2]|6 &
					7 & |[fill=t3]|2 &
					|[fill=t1]|4 & 1 \\
					|[fill=t5]|2 & |[fill=t4]|7 &
					|[fill=t3]|4 &
					0 & 9 &
					1 &
					|[fill=t2]|8 & 6 &
					3 & |[fill=t1]|5 \\
					};
				\draw[ultra thick, red] (A-1-1.north
				west)
				rectangle (A-10-4.south east);
				\draw[ultra thick, red] (A-1-1.north
				west)
				rectangle (A-4-10.south east);
				\draw[ultra thick] (A-1-1.north west)
				rectangle
				(A-10-10.south east);
			\end{scope}
			\begin{scope}[xshift=8cm]
				\matrix (A) [square freq small] {
					0 & 1 &
					2 &
					|[fill=t1]|3 & |[fill=t2]|4 &
					|[fill=t3]|5 &
					|[fill=t4]|6 & 7 &
					|[fill=t5]|8 & 9 \\
					1 & |[fill=t3]|8 &
					|[fill=t1]|6 &
					4 & |[fill=t5]|5 &
					2 &
					7 & |[fill=t2]|3 &
					9 & |[fill=t4]|0 \\
					|[fill=t2]|6 & |[fill=t1]|4 &
					|[fill=t4]|7 &
					5 & 8 &
					0 &
					|[fill=t5]|3 & 9 &
					1 & |[fill=t3]|2 \\
					|[fill=t1]|2 & |[fill=t2]|5 &
					3 &
					9 & |[fill=t3]|0 &
					7 &
					8 & |[fill=t4]|1 &
					6 & |[fill=t5]|4 \\
					|[fill=t4]|5 & 2 &
					|[fill=t5]|9 &
					|[fill=t2]|0 & |[fill=t1]|7 &
					8 &
					1 & 4 &
					|[fill=t3]|3 & 6 \\
					3 & 6 &
					0 &
					|[fill=t4]|8 & 9 &
					|[fill=t1]|1 &
					|[fill=t3]|4 & |[fill=t5]|2 &
					|[fill=t2]|7 & 5 \\
					4 & |[fill=t5]|0 &
					5 &
					|[fill=t3]|7 & 6 &
					|[fill=t4]|3 &
					|[fill=t1]|9 & 8 &
					2 & |[fill=t2]|1 \\
					|[fill=t3]|9 & 3 &
					|[fill=t2]|8 &
					2 & 1 &
					|[fill=t5]|6 &
					0 & |[fill=t1]|5 &
					|[fill=t4]|4 & 7 \\
					8 & 7 &
					4 &
					|[fill=t5]|1 & |[fill=t4]|2 &
					|[fill=t2]|9 &
					5 & |[fill=t3]|6 &
					|[fill=t1]|0 & 3 \\
					|[fill=t5]|7 & |[fill=t4]|9 &
					|[fill=t3]|1 &
					6 & 3 &
					4 &
					|[fill=t2]|2 & 0 &
					5 & |[fill=t1]|8 \\
					};
				\draw[ultra thick, red] (A-1-1.north
				west)
				rectangle (A-10-4.south east);
				\draw[ultra thick, red] (A-1-1.north
				west)
				rectangle (A-4-10.south east);
				\draw[ultra thick] (A-1-1.north west)
				rectangle
				(A-10-10.south east);
			\end{scope}
		\end{tikzpicture}
	\caption{\label{fig:4444with5trans1}A \net{4}{10} satisfying a
          relation of type $4^4$ with 7 common
          transversals and a set of 5 disjoint common transversals
          (coloured).}
	\end{center}
\end{figure}

\begin{figure}
	\begin{center}
		\begin{tikzpicture}
			\begin{scope}
				\matrix (A) [square freq small] {
				|[fill=t1]|0 & |[fill=t2]|1 &
				|[fill=t3]|2 &
				|[fill=t0]|3 & |[fill=t0]|4 &
				|[fill=t0]|5 &
				|[fill=t0]|6 & |[fill=t4]|7 &
				|[fill=t5]|8 & |[fill=t0]|9 \\
				|[fill=t3]|8 & |[fill=t5]|5 &
				|[fill=t0]|9 &
				|[fill=t0]|7 & |[fill=t4]|6 &
				|[fill=t2]|4 &
				|[fill=t0]|3 & |[fill=t0]|0 &
				|[fill=t1]|2 & |[fill=t0]|1 \\
				|[fill=t0]|7 & |[fill=t0]|8 &
				|[fill=t1]|4 &
				|[fill=t4]|1 & |[fill=t2]|9 &
				|[fill=t5]|2 &
				|[fill=t0]|5 & |[fill=t0]|3 &
				|[fill=t0]|0 & |[fill=t3]|6 \\
				|[fill=t4]|9 & |[fill=t0]|2 &
				|[fill=t0]|3 &
				|[fill=t5]|6 & |[fill=t3]|1 &
				|[fill=t1]|8 &
				|[fill=t2]|7 & |[fill=t0]|4 &
				|[fill=t0]|5 & |[fill=t0]|0 \\
				|[fill=t5]|4 & |[fill=t1]|6 &
				|[fill=t0]|7 &
				|[fill=t0]|9 & |[fill=t0]|5 &
				|[fill=t0]|1 &
				|[fill=t3]|0 & |[fill=t0]|2 &
				|[fill=t2]|3 & |[fill=t4]|8 \\
				|[fill=t0]|6 & |[fill=t0]|4 &
				|[fill=t0]|1 &
				|[fill=t0]|8 & |[fill=t5]|3 &
				|[fill=t4]|0 &
				|[fill=t1]|9 & |[fill=t3]|5 &
				|[fill=t0]|7 & |[fill=t2]|2 \\
				|[fill=t2]|5 & |[fill=t3]|7 &
				|[fill=t0]|6 &
				|[fill=t0]|2 & |[fill=t0]|0 &
				|[fill=t0]|9 &
				|[fill=t5]|1 & |[fill=t0]|8 &
				|[fill=t4]|4 & |[fill=t1]|3 \\
				|[fill=t0]|3 & |[fill=t0]|0 &
				|[fill=t2]|8 &
				|[fill=t3]|4 & |[fill=t1]|7 &
				|[fill=t0]|6 &
				|[fill=t4]|2 & |[fill=t5]|9 &
				|[fill=t0]|1 & |[fill=t0]|5 \\
				|[fill=t0]|2 & |[fill=t0]|9 &
				|[fill=t4]|5 &
				|[fill=t2]|0 & |[fill=t0]|8 &
				|[fill=t3]|3 &
				|[fill=t0]|4 & |[fill=t1]|1 &
				|[fill=t0]|6 & |[fill=t5]|7 \\
				|[fill=t0]|1 & |[fill=t4]|3 &
				|[fill=t5]|0 &
				|[fill=t1]|5 & |[fill=t0]|2 &
				|[fill=t0]|7 &
				|[fill=t0]|8 & |[fill=t2]|6 &
				|[fill=t3]|9 & |[fill=t0]|4 \\
				};
				\draw[ultra thick, red] (A-1-1.north
				west)
				rectangle (A-10-4.south east);
				\draw[ultra thick, red] (A-1-1.north
				west)
				rectangle (A-4-10.south east);
				\draw[ultra thick] (A-1-1.north west)
				rectangle
				(A-10-10.south east);
			\end{scope}
			\begin{scope}[xshift=8cm]
				\matrix (A) [square freq small] {
				|[fill=t1]|0 & |[fill=t2]|1 &
				|[fill=t3]|2 &
				|[fill=t0]|3 & |[fill=t0]|4 &
				|[fill=t0]|5 &
				|[fill=t0]|6 & |[fill=t4]|7 &
				|[fill=t5]|8 & |[fill=t0]|9 \\
				|[fill=t3]|5 & |[fill=t5]|7 &
				|[fill=t0]|6 &
				|[fill=t0]|4 & |[fill=t4]|2 &
				|[fill=t2]|3 &
				|[fill=t0]|0 & |[fill=t0]|1 &
				|[fill=t1]|9 & |[fill=t0]|8 \\
				|[fill=t0]|2 & |[fill=t0]|3 &
				|[fill=t1]|8 &
				|[fill=t4]|5 & |[fill=t2]|7 &
				|[fill=t5]|0 &
				|[fill=t0]|4 & |[fill=t0]|9 &
				|[fill=t0]|6 & |[fill=t3]|1 \\
				|[fill=t4]|4 & |[fill=t0]|6 &
				|[fill=t0]|1 &
				|[fill=t5]|9 & |[fill=t3]|0 &
				|[fill=t1]|7 &
				|[fill=t2]|8 & |[fill=t0]|5 &
				|[fill=t0]|2 & |[fill=t0]|3 \\
				|[fill=t5]|1 & |[fill=t1]|4 &
				|[fill=t0]|0 &
				|[fill=t0]|8 & |[fill=t0]|9 &
				|[fill=t0]|2 &
				|[fill=t3]|7 & |[fill=t0]|3 &
				|[fill=t2]|5 & |[fill=t4]|6 \\
				|[fill=t0]|3 & |[fill=t0]|2 &
				|[fill=t0]|7 &
				|[fill=t0]|0 & |[fill=t5]|6 &
				|[fill=t4]|9 &
				|[fill=t1]|5 & |[fill=t3]|8 &
				|[fill=t0]|1 & |[fill=t2]|4 \\
				|[fill=t2]|6 & |[fill=t3]|9 &
				|[fill=t0]|5 &
				|[fill=t0]|7 & |[fill=t0]|8 &
				|[fill=t0]|1 &
				|[fill=t5]|3 & |[fill=t0]|4 &
				|[fill=t4]|0 & |[fill=t1]|2 \\
				|[fill=t0]|7 & |[fill=t0]|5 &
				|[fill=t2]|9 &
				|[fill=t3]|6 & |[fill=t1]|3 &
				|[fill=t0]|8 &
				|[fill=t4]|1 & |[fill=t5]|2 &
				|[fill=t0]|4 & |[fill=t0]|0 \\
				|[fill=t0]|8 & |[fill=t0]|0 &
				|[fill=t4]|3 &
				|[fill=t2]|2 & |[fill=t0]|1 &
				|[fill=t3]|4 &
				|[fill=t0]|9 & |[fill=t1]|6 &
				|[fill=t0]|7 & |[fill=t5]|5 \\
				|[fill=t0]|9 & |[fill=t4]|8 &
				|[fill=t5]|4 &
				|[fill=t1]|1 & |[fill=t0]|5 &
				|[fill=t0]|6 &
				|[fill=t0]|2 & |[fill=t2]|0 &
				|[fill=t3]|3 & |[fill=t0]|7 \\
				};
				\draw[ultra thick, red] (A-1-1.north
				west)
				rectangle (A-10-4.south east);
				\draw[ultra thick, red] (A-1-1.north
				west)
				rectangle (A-4-10.south east);
				\draw[ultra thick] (A-1-1.north west)
				rectangle
				(A-10-10.south east);
			\end{scope}
		\end{tikzpicture}
	\caption{\label{fig:4444with5trans2}A \net{4}{10} satisfying a
          relation of type $4^4$ with 19 common
          transversals and a set of 5 disjoint common transversals
          (coloured).}
	\end{center}
\end{figure}

\section{Nets with Odd Relations of Type $4^22^3$}\label{s:44222}

While the set of templates of type $4^22^3$ is small, each template contains
three frequency squares with approximately $74\,000$ candidate lines per
frequency square. As the complexity of \Cref{alg:backtracking} is dependent on
the number of candidate lines, the time required to find all nets that are
extensions of the trivial partial net using the candidate lines is
prohibitively large. Since it is conjectured that no \net{5}{10} exists, the
computation of nets containing an odd relation of type $4^22^3$ is likely to
return no output. For these two reasons we instead enumerated a certain set
$\Omega'$ of \mols{2}{10}.  This set is such that if there is a \net{5}{10}
satisfying an odd relation of type $4^22^3$ then it must contain a subnet
isomorphic to a pair of MOLS in $\Omega'$. Given such a set, if no pair of MOLS
in $\Omega'$ is extendible to a triple, then no \net{5}{10} satisfies an odd
relation of type $4^22^3$. 

\begin{figure}
  \begin{center}
    \begin{tikzpicture}
      \begin{scope}
	\matrix (A) [square freq small] {
|[fill=t1]| 0 & |[fill=t0]| 1 & |[fill=t0]| 2 & |[fill=t2]| 3 & |[fill=t3]| 4 & |[fill=t4]| 5 & |[fill=t0]| 6 & |[fill=t0]| 7 & |[fill=t0]| 8 & |[fill=t0]| 9 \\
|[fill=t0]| 1 & |[fill=t2]| 0 & |[fill=t1]| 9 & |[fill=t0]| 7 & |[fill=t0]| 6 & |[fill=t0]| 3 & |[fill=t0]| 4 & |[fill=t4]| 8 & |[fill=t0]| 2 & |[fill=t3]| 5 \\
|[fill=t3]| 6 & |[fill=t1]| 8 & |[fill=t2]| 1 & |[fill=t0]| 0 & |[fill=t0]| 5 & |[fill=t0]| 9 & |[fill=t4]| 7 & |[fill=t0]| 2 & |[fill=t0]| 4 & |[fill=t0]| 3 \\
|[fill=t2]| 5 & |[fill=t0]| 4 & |[fill=t4]| 0 & |[fill=t1]| 1 & |[fill=t0]| 2 & |[fill=t3]| 8 & |[fill=t0]| 9 & |[fill=t0]| 3 & |[fill=t0]| 6 & |[fill=t0]| 7 \\
|[fill=t0]| 3 & |[fill=t0]| 5 & |[fill=t0]| 4 & |[fill=t4]| 6 & |[fill=t0]| 0 & |[fill=t0]| 1 & |[fill=t0]| 8 & |[fill=t3]| 9 & |[fill=t2]| 7 & |[fill=t1]| 2 \\
|[fill=t0]| 4 & |[fill=t0]| 3 & |[fill=t0]| 7 & |[fill=t3]| 2 & |[fill=t4]| 1 & |[fill=t0]| 0 & |[fill=t0]| 5 & |[fill=t1]| 6 & |[fill=t0]| 9 & |[fill=t2]| 8 \\
|[fill=t4]| 2 & |[fill=t0]| 6 & |[fill=t3]| 3 & |[fill=t0]| 9 & |[fill=t0]| 8 & |[fill=t0]| 7 & |[fill=t0]| 1 & |[fill=t2]| 4 & |[fill=t1]| 5 & |[fill=t0]| 0 \\
|[fill=t0]| 7 & |[fill=t4]| 9 & |[fill=t0]| 5 & |[fill=t0]| 8 & |[fill=t0]| 3 & |[fill=t1]| 4 & |[fill=t2]| 2 & |[fill=t0]| 0 & |[fill=t3]| 1 & |[fill=t0]| 6 \\
|[fill=t0]| 8 & |[fill=t3]| 7 & |[fill=t0]| 6 & |[fill=t0]| 5 & |[fill=t2]| 9 & |[fill=t0]| 2 & |[fill=t1]| 3 & |[fill=t0]| 1 & |[fill=t0]| 0 & |[fill=t4]| 4 \\
|[fill=t0]| 9 & |[fill=t0]| 2 & |[fill=t0]| 8 & |[fill=t0]| 4 & |[fill=t1]| 7 & |[fill=t2]| 6 & |[fill=t3]| 0 & |[fill=t0]| 5 & |[fill=t4]| 3 & |[fill=t0]| 1 \\
};
	\draw[ultra thick, red] (A-1-1.north
	west)
	rectangle (A-10-4.south east);
	\draw[ultra thick, red] (A-1-1.north
	west)
	rectangle (A-4-10.south east);
	\draw[ultra thick] (A-1-1.north west)
	rectangle
	(A-10-10.south east);
      \end{scope}
      \begin{scope}[xshift=8cm]
	\matrix (A) [square freq small] {
|[fill=t1]| 0 & |[fill=t0]| 1 & |[fill=t0]| 2 & |[fill=t2]| 3 & |[fill=t3]| 4 & |[fill=t4]| 5 & |[fill=t0]| 6 & |[fill=t0]| 7 & |[fill=t0]| 8 & |[fill=t0]| 9 \\
|[fill=t0]| 7 & |[fill=t2]| 2 & |[fill=t1]| 5 & |[fill=t0]| 0 & |[fill=t0]| 8 & |[fill=t0]| 6 & |[fill=t0]| 9 & |[fill=t4]| 4 & |[fill=t0]| 1 & |[fill=t3]| 3 \\
|[fill=t3]| 2 & |[fill=t1]| 9 & |[fill=t2]| 0 & |[fill=t0]| 4 & |[fill=t0]| 7 & |[fill=t0]| 8 & |[fill=t4]| 1 & |[fill=t0]| 3 & |[fill=t0]| 6 & |[fill=t0]| 5 \\
|[fill=t2]| 6 & |[fill=t0]| 0 & |[fill=t4]| 8 & |[fill=t1]| 2 & |[fill=t0]| 5 & |[fill=t3]| 1 & |[fill=t0]| 7 & |[fill=t0]| 9 & |[fill=t0]| 3 & |[fill=t0]| 4 \\
|[fill=t0]| 1 & |[fill=t0]| 8 & |[fill=t0]| 3 & |[fill=t4]| 7 & |[fill=t0]| 9 & |[fill=t0]| 4 & |[fill=t0]| 2 & |[fill=t3]| 0 & |[fill=t2]| 5 & |[fill=t1]| 6 \\
|[fill=t0]| 5 & |[fill=t0]| 4 & |[fill=t0]| 9 & |[fill=t3]| 8 & |[fill=t4]| 6 & |[fill=t0]| 3 & |[fill=t0]| 0 & |[fill=t1]| 1 & |[fill=t0]| 2 & |[fill=t2]| 7 \\
|[fill=t4]| 9 & |[fill=t0]| 5 & |[fill=t3]| 7 & |[fill=t0]| 6 & |[fill=t0]| 0 & |[fill=t0]| 2 & |[fill=t0]| 3 & |[fill=t2]| 8 & |[fill=t1]| 4 & |[fill=t0]| 1 \\
|[fill=t0]| 8 & |[fill=t4]| 3 & |[fill=t0]| 1 & |[fill=t0]| 5 & |[fill=t0]| 2 & |[fill=t1]| 7 & |[fill=t2]| 4 & |[fill=t0]| 6 & |[fill=t3]| 9 & |[fill=t0]| 0 \\
|[fill=t0]| 3 & |[fill=t3]| 6 & |[fill=t0]| 4 & |[fill=t0]| 9 & |[fill=t2]| 1 & |[fill=t0]| 0 & |[fill=t1]| 8 & |[fill=t0]| 5 & |[fill=t0]| 7 & |[fill=t4]| 2 \\
|[fill=t0]| 4 & |[fill=t0]| 7 & |[fill=t0]| 6 & |[fill=t0]| 1 & |[fill=t1]| 3 & |[fill=t2]| 9 & |[fill=t3]| 5 & |[fill=t0]| 2 & |[fill=t4]| 0 & |[fill=t0]| 8 \\
	};
	\draw[ultra thick, red] (A-1-1.north
	west)
	rectangle (A-10-4.south east);
	\draw[ultra thick, red] (A-1-1.north
	west)
	rectangle (A-4-10.south east);
	\draw[ultra thick] (A-1-1.north west)
	rectangle
	(A-10-10.south east);
      \end{scope}
    \end{tikzpicture}
    \caption{\label{fig:44222with4trans1}A \net{4}{10} contained
      in $\Omega'$ with 4 common transversals that are all disjoint (coloured) obtained from
      \Cref{fig:44222template}. The $0$ and $1$ symbol lines are
      relational in the first square and the $0$ and $2$ symbol
      lines are relational in the second.}
  \end{center}
\end{figure}

\begin{figure}
	\begin{center}
		\begin{tikzpicture}
			
			\begin{scope}
				\matrix (A) [square freq small] {
|[fill=t1]|0 & |[fill=t0]|1 & |[fill=t0]|2 & |[fill=t0]|3 & |[fill=t0]|4 & |[fill=t2]|5 & |[fill=t3]|6 & |[fill=t0]|7 & |[fill=t0]|8 & |[fill=t0]|9 \\
|[fill=t0]|6 & |[fill=t1]|2 & |[fill=t0]|8 & |[fill=t2]|0 & |[fill=t0]|7 & |[fill=t3]|3 & |[fill=t0]|9 & |[fill=t0]|5 & |[fill=t0]|1 & |[fill=t0]|4 \\
|[fill=t0]|2 & |[fill=t0]|4 & |[fill=t0]|0 & |[fill=t1]|9 & |[fill=t0]|6 & |[fill=t0]|8 & |[fill=t0]|5 & |[fill=t2]|1 & |[fill=t0]|3 & |[fill=t3]|7 \\
|[fill=t2]|7 & |[fill=t0]|0 & |[fill=t1]|5 & |[fill=t3]|2 & |[fill=t0]|9 & |[fill=t0]|1 & |[fill=t0]|3 & |[fill=t0]|4 & |[fill=t0]|6 & |[fill=t0]|8 \\
|[fill=t3]|5 & |[fill=t2]|6 & |[fill=t0]|9 & |[fill=t0]|8 & |[fill=t1]|3 & |[fill=t0]|7 & |[fill=t0]|0 & |[fill=t0]|2 & |[fill=t0]|4 & |[fill=t0]|1 \\
|[fill=t0]|1 & |[fill=t0]|9 & |[fill=t0]|7 & |[fill=t0]|6 & |[fill=t2]|8 & |[fill=t1]|4 & |[fill=t0]|2 & |[fill=t0]|3 & |[fill=t3]|0 & |[fill=t0]|5 \\
|[fill=t0]|8 & |[fill=t0]|7 & |[fill=t3]|4 & |[fill=t0]|5 & |[fill=t0]|0 & |[fill=t0]|2 & |[fill=t1]|1 & |[fill=t0]|6 & |[fill=t2]|9 & |[fill=t0]|3 \\
|[fill=t0]|4 & |[fill=t0]|5 & |[fill=t2]|3 & |[fill=t0]|1 & |[fill=t0]|2 & |[fill=t0]|6 & |[fill=t0]|8 & |[fill=t3]|9 & |[fill=t1]|7 & |[fill=t0]|0 \\
|[fill=t0]|9 & |[fill=t0]|3 & |[fill=t0]|6 & |[fill=t0]|4 & |[fill=t3]|1 & |[fill=t0]|0 & |[fill=t0]|7 & |[fill=t1]|8 & |[fill=t0]|5 & |[fill=t2]|2 \\
|[fill=t0]|3 & |[fill=t3]|8 & |[fill=t0]|1 & |[fill=t0]|7 & |[fill=t0]|5 & |[fill=t0]|9 & |[fill=t2]|4 & |[fill=t0]|0 & |[fill=t0]|2 & |[fill=t1]|6 \\
};
				\draw[ultra thick, red] (A-1-1.north
				west)
				rectangle (A-10-4.south east);
				\draw[ultra thick, red] (A-1-1.north
				west)
				rectangle (A-4-10.south east);
				\draw[ultra thick] (A-1-1.north west)
				rectangle
				(A-10-10.south east);
			\end{scope}

			\begin{scope}[xshift=8cm]
				\matrix (A) [square freq small] {
|[fill=t1]|0 & |[fill=t0]|1 & |[fill=t0]|2 & |[fill=t0]|3 & |[fill=t0]|4 & |[fill=t2]|5 & |[fill=t3]|6 & |[fill=t0]|7 & |[fill=t0]|8 & |[fill=t0]|9 \\
|[fill=t0]|5 & |[fill=t1]|3 & |[fill=t0]|0 & |[fill=t2]|9 & |[fill=t0]|1 & |[fill=t3]|7 & |[fill=t0]|8 & |[fill=t0]|2 & |[fill=t0]|4 & |[fill=t0]|6 \\
|[fill=t0]|6 & |[fill=t0]|0 & |[fill=t0]|3 & |[fill=t1]|7 & |[fill=t0]|9 & |[fill=t0]|4 & |[fill=t0]|1 & |[fill=t2]|8 & |[fill=t0]|5 & |[fill=t3]|2 \\
|[fill=t2]|3 & |[fill=t0]|8 & |[fill=t1]|4 & |[fill=t3]|0 & |[fill=t0]|5 & |[fill=t0]|6 & |[fill=t0]|9 & |[fill=t0]|1 & |[fill=t0]|2 & |[fill=t0]|7 \\
|[fill=t3]|8 & |[fill=t2]|7 & |[fill=t0]|6 & |[fill=t0]|1 & |[fill=t1]|2 & |[fill=t0]|9 & |[fill=t0]|4 & |[fill=t0]|5 & |[fill=t0]|3 & |[fill=t0]|0 \\
|[fill=t0]|9 & |[fill=t0]|2 & |[fill=t0]|5 & |[fill=t0]|4 & |[fill=t2]|6 & |[fill=t1]|8 & |[fill=t0]|7 & |[fill=t0]|0 & |[fill=t3]|1 & |[fill=t0]|3 \\
|[fill=t0]|2 & |[fill=t0]|4 & |[fill=t3]|9 & |[fill=t0]|6 & |[fill=t0]|7 & |[fill=t0]|1 & |[fill=t1]|5 & |[fill=t0]|3 & |[fill=t2]|0 & |[fill=t0]|8 \\
|[fill=t0]|7 & |[fill=t0]|9 & |[fill=t2]|1 & |[fill=t0]|2 & |[fill=t0]|8 & |[fill=t0]|0 & |[fill=t0]|3 & |[fill=t3]|4 & |[fill=t1]|6 & |[fill=t0]|5 \\
|[fill=t0]|1 & |[fill=t0]|6 & |[fill=t0]|8 & |[fill=t0]|5 & |[fill=t3]|3 & |[fill=t0]|2 & |[fill=t0]|0 & |[fill=t1]|9 & |[fill=t0]|7 & |[fill=t2]|4 \\
|[fill=t0]|4 & |[fill=t3]|5 & |[fill=t0]|7 & |[fill=t0]|8 & |[fill=t0]|0 & |[fill=t0]|3 & |[fill=t2]|2 & |[fill=t0]|6 & |[fill=t0]|9 & |[fill=t1]|1 \\
};
				\draw[ultra thick, red] (A-1-1.north
				west)
				rectangle (A-10-4.south east);
				\draw[ultra thick, red] (A-1-1.north
				west)
				rectangle (A-4-10.south east);
				\draw[ultra thick] (A-1-1.north west)
				rectangle
				(A-10-10.south east);
			\end{scope}
		\end{tikzpicture}
	\caption{\label{fig:44222with4trans2}A \net{4}{10} that is a refinement of
          the the second and third squares of the template in
		  \Cref{fig:44222template}. The pair is not contained in $\Omega'$ and
          contains $13$ common transversals and a set of $3$ disjoint
          common transversals (coloured).  The $0$ and $2$ symbol
          lines are relational in the first square and the $0$ and $3$
          symbol lines are relational in the second.}
	\end{center}
\end{figure}

As $\Omega'$ will contain pairs of MOLS, we need to select $2$ frequency
squares and search for all possible refinements of these frequency squares. The
choice of the frequency squares affects not just the efficiency of computing
$\Omega'$, it affects the content of the set itself. For example, consider
\Cref{fig:44222with4trans1,fig:44222with4trans2}, both of which show pairs of
MOLS that are refinements of the frequency squares from the template in
\Cref{fig:44222template}. \Cref{fig:44222with4trans1} is a refinement of the
first and second frequency squares, whereas \Cref{fig:44222with4trans2} is a
refinement of the second and third frequency squares. As a consequence, it will
turn out that the MOLS in \Cref{fig:44222with4trans1} were in our set
$\Omega'$, but the MOLS in \Cref{fig:44222with4trans2} were not paratopic to
any pair of MOLS in $\Omega'$.

To construct the set $\Omega'$ we first selected an ordering of the three
frequency squares in each template.  Then we determined the set of candidate
lines with respect to the template.  Next, we used \Cref{alg:symmetrybroken} to
extend the trivial partial net to obtain pairs of MOLS that refine the first
and second frequency squares.  The construction of the orbits for the candidate
lines, as described in \Cref{section:sym}, is heavily dependent on the parallel
class $\Pi_2$. This is one reason why the ordering of the squares in the
template affects the efficiency of the computation.

All search strategies supplied to \Cref{alg:symmetrybroken} first extended the
trivial partial net by selecting the two relational candidate lines in $\Pi_4$;
this is fast and greatly restricts the lines that are available for $\Pi_2$ and
$\Pi_3$. The remaining details of the search strategy depended on the
automorphism group of the template concerned, and will be discussed in detail
below. The result of this computation is a set of partial $5$-nets of order
$10$. The partial parallel class $\Pi_4$ was then discarded to give a set of
pairs of MOLS.  This set was screened to obtain a set $\Omega'(T)$ of species
representatives of the MOLS that we obtained from the template $T$.  Finally,
the set $\Omega'$ was obtained by taking the union over all sets $\Omega'(T)$
for the 30 choices of template $T$ given in \Tref{table:templatecompletions}.
It is clear from this process that if there exists a \net{5}{10} $N$ satisfying
an odd relation of type $4^22^3$ then there is a \net{4}{10} in $\Omega'$ that
is isomorphic to a subnet of $N$.

\subsection{Symmetries in Templates}\label{s:44222sym}

We say that a template of type $4^22^3$ has \emph{$S_3$ symmetry} if the
automorphism group of the template can induce any reordering of the three
frequency squares in the template. A template admits a \emph{$C_2$ symmetry} if
there is an automorphism that exchanges two of the frequency squares but the
template does not have $S_3$ symmetry. Of the $30$ templates of type $4^22^3$,
there are $9$ possessing $S_3$ symmetry, $17$ with a $C_2$ symmetry and the
remaining $4$ templates admit no automorphism that reorders the frequency
squares.

\Cref{table:templatecompletions} lists the templates that we used to compute
$\Omega'$. For templates other than those with $C_2$ symmetry we had no reason
to choose any ordering of the frequency squares over any other ordering, so we
simply used them in the order that we had generated them. The situation for
$C_2$ symmetry is more interesting.  After some experimenting we discovered
that it is best to use an ordering for which the $C_2$ symmetry exchanges the
second and third frequency squares. However, we did not realise this until
after we had computed all nets (up to isomorphism) that are refinements of the
template marked with a $^*$ in \Cref{table:templatecompletions}, so for that
template only we had the $C_2$ symmetry exchanging the first two frequency
squares.

The reason it is better to have the $C_2$ exchanging the last two squares is
related to the size and position of the orbits of candidate lines.  Each of
these orbits typically includes twice as many lines in the square that is fixed
by the $C_2$ compared to each of the other two squares (individually). We found
it was quicker to have larger orbits in the first square compared to orbits of
the same size split between the first two squares. The worst option of all is
to position the $C_2$ so that it exchanges the first and third square, since
then half of each orbit is wasted because of the fact that we only ever add two
lines to $\Pi_4$.

We used two different search strategies. Both found extensions of the trivial
partial net by selecting two relational candidate lines in $\Pi_4$ first, but
after that they diverged. One strategy proceeded by alternating between the
first two frequency squares. The other found a refinement of the first square
to a parallel class before turning to the second frequency square. Our limited
experimentation suggests that the alternating strategy is faster whenever the
two squares have similar numbers of candidate lines remaining. This is always
true at the start of the search, and it remains true throughout the search when
the template has $S_3$ symmetry. For templates without $S_3$ symmetry (assuming
we implement our preference above in the case of $C_2$ symmetry), the number of
candidate lines becomes unbalanced as orbits are discarded from the first
square. At some point it becomes faster to switch to the non-alternating search
strategy. The exact point at which this happens is unclear and was only roughly
approximated. Our two implementations chose to switch strategies at different
points.

\subsection{Computational Results}

Our computation of $\Omega'$ found $100\,826$ species of \mols{2}{10}.  For
each pair of MOLS the common transversals were computed. No pair of MOLS had
more than four disjoint common transversals. As the computation of $\Omega'$
began by adding the two relational lines to $\Pi_4$, each pair of MOLS has at
least two disjoint common transversals. \Cref{table:44222transversals} lists
the number of pairs of MOLS classified by number of common transversals and
disjoint common transversals. Every pair of MOLS in $\Omega'$ has dimension 37,
satisfying no non-trivial relations. For each pair of MOLS in $\Omega'$ we
attempted to extend the pair to a triple in all possible ways (not just those
consistent with the third frequency square).  No extension was found, giving
our second main result:

\begin{theorem}\label{t:no44222}
	Let $N$ be a \net{5}{10}. Then $N$ does not satisfy an odd
        relation of type $2^34^2$.
\end{theorem}

\begin{table}[h]
	
  \centering
	
  \small{
    \begin{tabular}{ccr}
      \toprule
      \multicolumn{1}{l}{Disjoint}     & \multicolumn{1}{l}{}             & \multicolumn{1}{l}{}       \\
      \multicolumn{1}{l}{Transversals} & \multicolumn{1}{l}{Transversals} & \multicolumn{1}{l}{Species}  \\
      \midrule
      2                                & 2                                & 88\,611                   \\
      2                                & 3                                & 8\,317 \\
      2                                & 4                                & 401\\
      2                                & 5                                & 11\vspace{1mm}\\
      3                                & 3                                & 3\,062\\
      3                                & 4                                & 387\\
      3                                & 5                                & 11\\
      3                                & 6                                & 1\\
      3                                & 8                                & 1\vspace{1mm}\\
      4                                & 4                                & 22\\
      4                                & 5                                & 2                       
    \end{tabular}
    
  }

  \caption{\label{table:44222transversals}The MOLS
          in $\Omega'$ classified by the number of
          transversals and disjoint transversals.}
\end{table}

\section{Concluding remarks}\label{s:conclude}

We have enumerated all \mols{2}{10} that satisfy a non-trivial relation and
found that none of them extend to \mols{3}{10}.  In \Cref{t:no44222}, we have
also eliminated the possibility of \mols{3}{10} satisfying a relation of type
22246. There are nine other plausible types of relations on \mols{4}{10} listed
by \citet{dukes_group_2014}. If all nine could be eliminated then, by
\Cref{t:DH}, it would follow that $N(10)\le3$.  Unfortunately, that is not
possible with our method. Based on evidence including \Cref{table:solcounts} it
seems a general rule that relations have greater flexibility (and hence require
more computation to eliminate) if they have more of their $\lambda_i$ closer to
$n/2$.  In other words, we have eliminated only the easiest of the plausible
relations on \mols{3}{10}, and we believe type $4^5$ to be very much harder to
eliminate than type $22246$.  Furthermore, we expect that eliminating relations
on \mols{4}{10} will be significantly harder again.

Nevertheless, we have made a useful start. As a glimpse of the possibilities
our work might open up, consider the following situation. Suppose that $N_8$ is
an \net{8}{10} and that $N_6$ is the sub $6$-net of largest dimension. Let
$N_7$ be either of the sub $7$-nets of $N_8$ that contain $N_6$.  By
\cite[Prop.~2.5]{dukes_group_2014} we know that $\dim(N_7)\le53$.  However,
$N_7$ is not maximal, so by \cite[Cor.~2.2]{dougherty_coding_1994-1}, we have
$\dim(N_6)<\dim(N_7)$. This means that $N_6$ has at least $t$ independent
non-trivial relations on it, for some $t\ge3$.  By choice of $N_6$, every sub
$6$-net of $N_8$ has the same property.

By \Cref{thm:main_4444} we know that no sub $4$-net of $N_8$ has a non-trivial
relation. Each of the $\binom{8}{6}=28$ sub $6$-nets of $N_8$ has $t$
independent non-trivial relations on it, and each such relation appears on at
most three different sub $6$-nets. Hence, we have identified at least $28t/3$
different non-trivial relations on $N_8$.  We can see no reason why these
relations should be independent. However, if they were, it would show that
\[55-t=\dim(N_6)\le\dim(N_8)\le73-28t/3.\] This would contradict $t\ge3$,
showing that $N_8$ could not exist and demonstrating that $N(10)\le5$.

Finally, we note that the MOLS in $\Omega$ and $\Omega'$ can be downloaded from
\cite{WWWW}, as can the templates used to generate them.

\section*{Acknowledgements}

This research was supported by the Monash eResearch Centre and
eSolutions-Research Support Services through the use of the MonARCH HPC
Cluster. The second author is extremely grateful for the generous hospitality
of Wendy Myrvold and Peter Dukes, who taught him about relations on nets during
his visits to University of Victoria, BC.

 
  \let\oldthebibliography=\thebibliography
  \let\endoldthebibliography=\endthebibliography
  \renewenvironment{thebibliography}[1]{%
    \begin{oldthebibliography}{#1}%
      \setlength{\parskip}{0.2ex}%
      \setlength{\itemsep}{0.2ex}%
  }%
  {%
    \end{oldthebibliography}%
  }


\begin{thebibliography}{99}
	\providecommand{\natexlab}[1]{#1}
	\providecommand{\url}[1]{\texttt{#1}}
	\expandafter\ifx\csname urlstyle\endcsname\relax
	\providecommand{\doi}[1]{doi: #1}\else
	\providecommand{\doi}{doi: \begingroup \urlstyle{rm}\Url}\fi
	
	\bibitem[Best(2018)]{best_thesis_2018}
	D.~Best.
	\newblock Transversal this, transversal that.
	\newblock PhD thesis, Monash University, 2018.

	\bibitem[Bose and Shrikhande(1959)]{bose_falsity_1959}
	R.\,C.~Bose and S.\,S.~Shrikhande.
	\newblock On the falsity of {{Euler}}'s conjecture about the non-existence
	  of two orthogonal {{Latin}} squares of order $4t+ 2$.
	\newblock \emph{Proc. Nat. Acad. Sci. U.S.A.}, 45:734--737, 1959.
	
	\bibitem[Bose et~al.(1960)Bose, Shrikhande, and Parker]{bose_further_1960}
	R.\,C.~Bose, S.\,S.~Shrikhande, and E.\,T.~Parker.
	\newblock {F}urther results on the construction of mutually orthogonal
	  {L}atin squares and the falsity of {{Euler}}'s conjecture.
	\newblock \emph{Canadian J. Math.}, 12:189--203, 1960.
	
	\bibitem[Brown et~al.(1993)Brown, Hedayat, and Parker]{brown_pair_1993}
	J.\,W.~Brown, A.\,S.~Hedayat, and E.\,T.~Parker.
	\newblock A pair of orthogonal {L}atin squares of order 10 with four shared
	  parallel transversals.
	\newblock \emph{J. Combin. Inform. System Sci.}, 18:1--2, 1993.
	
	\bibitem[Bruck(1963)]{bruck_finite_1963}
	R.\,H.~Bruck.
	\newblock Finite nets. {{II}}. {{Uniqueness}} and imbedding.
	\newblock \emph{Pacific J. Math.}, 13:421--457, 1963.
	
	\bibitem[Delisle(2006)]{delisle_masters_2006}
	E.~Delisle.
	\newblock The search for a triple of mutually orthogonal {L}atin
	  squares of order ten: Looking through pairs of dimension
	  thirty-five of less.
	\newblock Master's thesis, University of Victoria, 2006.
	
	\bibitem[Dougherty(1994)]{dougherty_coding_1994-1}
	S.\,T.~Dougherty.
	\newblock {A} coding theoretic solution to the 36 officer problem.
	\newblock \emph{Des. Codes Cryptogr.}, 4:123--128, 1994.
	
	\bibitem[Dukes and Howard(2014)]{dukes_group_2014}
	P.~Dukes and L.~Howard.
	\newblock Group divisible designs in {{MOLS}} of order ten.
	\newblock \emph{Des. Codes Cryptogr.}, 71:283--291, 2014.
	
	\bibitem[Egan and Wanless(2016)]{egan_enumeration_2016}
	J.~Egan and I.\,M.~Wanless.
	\newblock Enumeration of MOLS of small order.
	\newblock \emph{Math. Comp.}, 85:799--824, 2016.
	
	\bibitem[Lam et~al.(1989)Lam, Thiel, and Swiercz]{lam_non-existence_1989}
	C.\,W.\,H.~Lam, L.~Thiel, and S.~Swiercz.
	\newblock The nonexistence of finite projective planes of order 10.
	\newblock \emph{Canadian J. Math.}, 41:1117--1123, 1989.
	
	\bibitem[Mathon(1997)]{mathon_searching_1997}
	R.~Mathon.
	\newblock Searching for spreads and packings.
	\newblock \emph{London Math. Soc. Lecture Note Ser.},
        245:161--176, 1997.
	  
	\bibitem[McKay, Meynert and Myrvold(2007)]{mckay_small_2007}
	B.\,D.~McKay, A.~Meynert, and W.~Myrvold.
	\newblock Small Latin squares, quasigroups, and loops.
	\newblock \emph{J. Combin. Des.}, 15:98--119,
	2007.

	\bibitem[Metsch(1991)]{metsch_improvement_1991}
	K.~Metsch.
	\newblock Improvement of {B}ruck's completion theorem.
	\newblock \emph{Des. Codes Cryptogr.}, 1:99--116, 1991.
	
	\bibitem[Parker(1959)]{parker_construction_1959}
	E.\,T.~Parker.
	\newblock Construction of some sets of mutually orthogonal {L}atin squares.
	\newblock \emph{Proc. Amer. Math. Soc.},
        10:946--949, 1959.
	
	\bibitem[Stinson(1984)]{stinson_short_1984}
	D.\,R.~Stinson.
	\newblock A short proof of the nonexistence of a pair of orthogonal {L}atin
	  squares of order six.
	\newblock \emph{J. Combin. Theory Ser. A}, 36:373--376, 1984.
	
	\bibitem[Tarry(1901)]{tarry_probleme_1901}
	G.~Tarry.
	\newblock Le probl{\`e}me des 36 officiers.
	\newblock \emph{Compte Rendu de l'Assoc. Fran{\c{c}}aise Avanc. Sci.
	Naturel}, 2:170--203, 1901.

\bibitem{WWWW}
  I.\,M.~Wanless, Author's homepage,
  \url{http://users.monash.edu.au/~iwanless/data/MOLS/}
  
\end{thebibliography}
\end{document}